\documentclass[10pt]{amsart}

\usepackage{amsmath,amsthm,amssymb,color,fullpage,graphicx,appendix,multirow,subfig,longtable,bbold,mathabx,latexsym,setspace,pinlabel,url}
\usepackage[top=1in,bottom=1in,left=1in,right=1in]{geometry} 
\input xypic

\def\A{\mathcal{A}_R} 
\def\Az{\mathcal{A}_{\integers/2}}
\def\AZ{\mathcal{A}_Z}
\def\calA{\mathcal{A}}
\def\calE{\mathcal{E}}
\newcommand{\aug}{\operatorname{Aug}}
\newcommand{\s}{\operatorname{sgn}}

\def\integers{{\mathbb{Z}}}
\def\reals{{\mathbb R}}

\def\ker{{\hbox{ker}}\,}
\def\other{\text{otherwise}}
\newcommand{\st}{\ensuremath{\text{ s.t. }}}

\def\bigno{\bigskip\noindent} 
 
\def\medno{\medskip\noindent}

\newcommand{\diag}[1]{\begin{center}\begin{minipage}{5in}\xymatrix{#1}\end{minipage}\end{center}}
\newcommand{\seq}[1]{\begin{center}\begin{minipage}{5in}\xymatrix@R=.125pc{#1}\end{minipage}\end{center}}
 
\def\arr{\ar[r]} 
\def\ard{\ar[d]}

\def\mapstol{\ar@{|->}[l]} 
\def\mapstor{\ar@{|->}[r]}
\def\mapstod{\ar@{|->}[d]}
\def\mapstou{\ar@{|->}[u]}
\newtheorem{lem}{Lemma}[section]
\newtheorem{thm}[lem]{Theorem}

\newtheorem*{thmRhoOdd}{Theorem \ref{thm:rhoOdd}}
\newtheorem*{thmLift}{Theorem \ref{thm:lift}}

\theoremstyle{definition} 
\newtheorem{defn}[lem]{Definition}

\newtheorem{rmk}[lem]{Remark} 
\newtheorem{note}[lem]{Notation}

\onehalfspace

\begin{document}
\title{Augmentations and Rulings of Legendrian Knots}
\author{C. Leverson}
\address{Duke University, Durham, NC 27708}
\email{cleverso@math.duke.edu}
\date{\today}
\maketitle

\begin{abstract}
  For any Legendrian knot $\Lambda$ in $(\reals^3,\ker(dz-ydx))$, we
  show that the existence of an augmentation to any field of the
  Chekanov-Eliashberg differential graded algebra over
  $\integers[t,t^{-1}]$ is equivalent to the existence of a ruling of
  the front diagram, generalizing results of Fuchs, Ishkhanov, and
  Sabloff. We also show that any even graded augmentation must send
  $t$ to $-1$.
\end{abstract}

\section{Introduction}
A Legendrian knot in $(\reals^3,\xi_{\text{std}})$ is an embedding
$\Lambda:S^1\to\reals^3$ which is everywhere tangent to the contact
planes. In \cite{Chekanov} (see related \cite{EliashbergInvariants}),
Chekanov introduced a combinatorial way to associate a non-commutative
differential graded algebra (DGA) over $\integers/2$ to a Lagrangian
diagram of a Legendrian knot $\Lambda$ in $\reals^3$. The DGA is
generated by crossings of $\Lambda$ and the differential is determined
by a count of immersed polygons whose edges lie on the knot and whose
corners lie at crossings of $\Lambda$. In the literature, this DGA is
called the Chekanov-Eliashberg DGA. Chekanov showed that the homology
of the DGA is invariant under Legendrian isotopy. He also showed that
a linearized version of the homology of the DGA could be used to
distinguish between two Legendrian $5_2$ knots in $\reals^3$ which
could not be distinguished by the rotation and Thurston-Bennequin
numbers. In the early 2000's, Etnyre, Ng, and Sabloff gave a lift of
the Chekanov-Eliashberg DGA to a DGA $(\A,\partial)$ over
$R=\integers[t,t^{-1}]$ which has a full $\integers$-grading (see
\cite{EtnyreInvariants}). One can recover the Chekanov-Eliashberg DGA
by setting $t=1$, which requires one to consider the grading mod
$2r(\Lambda)$, and considering the coefficients mod $2$ (where
$r(\Lambda)$ is the rotation number, defined in \S
\ref{sec:background}). 

Another Legendrian knot invariant uses generating families, functions
whose critical values generate front diagrams of Legendrian
knots. Following ideas introduced by Eliashberg in
\cite{EliashbergWave}, Fuchs \cite{FuchsAug} and Chekanov-Pushkar
\cite{ChekanovFronts} gave invariants involving decompositions of the
generating families, which are now called ``normal rulings'' and can
also be used to distinguish between Chekanov's $5_2$ knots.

Remarkably, there is a close connection between the
Chekanov-Eliashberg DGA and rulings. Fuchs \cite{FuchsAug},
Fuchs-Ishkhanov \cite{FuchsIshkhanov}, and Sabloff \cite{SabloffAug}
showed that the existence of a ruling is equivalent to the existence
of an augmentation to $\integers/2$ of the Chekanov-Eliashberg DGA,
where an augmentation to a ring $S$ is an algebra map $\epsilon:\A\to
S$ such that $\epsilon\circ\partial=0$ and $\epsilon(1)=1$.

The main result of this paper gives a generalization of these results
using an extension of Sabloff's construction in \cite{SabloffAug}. Let
$F$ be a field and $R=\integers[t,t^{-1}]$. Given a $\rho$-graded
augmentation $\epsilon:\A\to F$ of the
$\integers[t,t^{-1}]$-differential graded algebra $(\A,\partial)$ of a
knot $\Lambda$, we will find a $\rho$-graded normal ruling of the knot
diagram.  Conversely, given a $\rho$-graded normal ruling of the knot
diagram, we will define a $\rho$-graded augmentation $\epsilon:\A\to
F$ of the DGA over $\integers[t,t^{-1}]$ with $\epsilon(t)=-1$. (For
$\rho=0$, this is the so called graded case and for $\rho=1$, the
ungraded case.) Terminology will be introduced in \S
\ref{sec:background}.

In \S \ref{sec:augRuling} and \S \ref{sec:rulingAug}, we will show:

\begin{thm}\label{thm:main}
  Let $\Lambda$ be a Legendrian knot in $\reals^3$. Given a field $F$,
  $(\A,\partial)$ has a $\rho$-graded augmentation $\epsilon:\A\to F$
  if and only if any front diagram of $\Lambda$ has a $\rho$-graded
  normal ruling. Furthermore, if $\rho$ is even, then
  $\epsilon(t)=-1$.
\end{thm}

Note that this generalizes Fuchs, Fuchs-Ishkhanov, and Sabloff's
results, giving a correspondence between normal rulings and
augmentations to any field $F$ of the DGA over
$\integers[t,t^{-1}]$. This does not contradict the result in
\cite{NgSatellites} that there are augmentations to matrix algebras
which do not send $t$ to $-1$ as the matrix algebras are not fields.

Theorem \ref{thm:main} can be extended and interpreted in terms of the
augmentation variety for a Legendrian knot. Define
\[\aug_\rho(\Lambda)=\{\epsilon(t):\epsilon\text{ a $\rho$-graded
  augmentation of }(\A,\partial)\}\subset F^*\] the {\bf augmentation
  variety} of $\Lambda$, where $F^*=F\backslash\{0\}$.

In higher dimensions, understanding the augmentation variety is
interesting and useful (see \cite{Aganagic} and \cite{NgFramed}), so
there has been some question as to whether we can determine the
augmentation variety in $\reals^3$ with the standard contact
structure. In \S \ref{sec:augRuling}, we prove:

\begin{thm}\label{thm:rhoOdd}
  If $\rho$ is odd and $\rho\vert2r(\Lambda)$, then
  \[\aug_\rho(\Lambda)=\begin{cases}
    \{-x^2:x\in F^*\}&\text{ if there exists a }\rho\text{-graded normal ruling of $\Lambda$ which is not oriented (introduced in \S\ref{sec:augRuling}})\\
    \{-1\}&\text{ if there exists a }\rho\text{-graded normal ruling of }\Lambda\text{ and all rulings are oriented}\\
    \emptyset&\text{ if there are no }\rho\text{-graded normal rulings
      of }\Lambda.
  \end{cases}\]
\end{thm}

For example, the right handed trefoil $\Lambda$ in Figure
\ref{fig:trefoilsEx} has DGA $(\A,\partial)$ with $\lvert c_i\rvert=0$
for $1\leq i\leq 3$, $\lvert c_4\rvert=\lvert c_5\rvert=1$, and $\lvert t\rvert=0$. Then
$\A=\A(c_1,\ldots,c_5)$ with differential
\begin{align*}
  \partial c_1&=\partial c_2=\partial c_3=0\\
  \partial c_4&=t+c_1+c_3+c_1c_2c_3\\
  \partial c_5&=1-c_1-c_3-c_3c_2c_1.
\end{align*}
Let $F$ be a field. If $\epsilon:\A\to F$ is a $1$-graded (ungraded)
augmentation, then
\begin{align*}
  0&=\epsilon(t)+\epsilon(c_1)+\epsilon(c_3)+\epsilon(c_1)\epsilon(c_2)\epsilon(c_3)\\
  0&=1-\epsilon(c_1)-\epsilon(c_3)-\epsilon(c_3)\epsilon(c_2)\epsilon(c_1)
\end{align*}
and so $\epsilon(t)=-1$. Thus $\aug_1(\Lambda)=\{-1\}$.

Now consider the left handed trefoil $\Lambda'$ depicted in Figure
\ref{fig:trefoilsEx}. The associated DGA is $(\A',\partial')$ with
$\lvert c_1\rvert=\lvert c_2\rvert=\lvert c_4\rvert=-1$, $\lvert
c_3\rvert=\lvert c_5\rvert=\lvert c_6\rvert=1$, and $\lvert
t\rvert=2$. Then $\A=\A(c_1,\ldots,c_6)$ with differential
\begin{align*}
  \partial' c_1&=\partial'c_2=\partial'c_3=0\\
  \partial' c_4&=t+c_1c_2\\
  \partial' c_5&=1+c_2c_3\\
  \partial' c_6&=1+c_3c_1.
\end{align*}
Let $F$ be a field. If $\epsilon:\A'\to F$ is a $1$-graded (ungraded)
augmentation, then
\begin{align*}
  0&=\epsilon(t)+\epsilon(c_1)\epsilon(c_2)\\
  0&=1+\epsilon(c_2)\epsilon(c_3)\\
  0&=1+\epsilon(c_3)\epsilon(c_1).
\end{align*}
Therefore $\epsilon(c_2)=-(\epsilon(c_3))^{-1}=\epsilon(c_1)$ and so
$\epsilon(t)=-(\epsilon(c_3))^{-2}$. So any nonzero choice of
$\epsilon(c_3)$ yields an augmentation and thus
$\aug_1(\Lambda')=\{-x^2:x\in F^*\}$.

\begin{figure}
  \labellist
  \small
  \pinlabel $c_1$ [b] at 44 288
  \pinlabel $c_2$ [b] at 256 290
  \pinlabel $c_3$ [b] at 447 289
  \pinlabel $c_4$ [b] at 543 389
  \pinlabel $c_5$ [t] at 546 160

  \pinlabel $c_1$ [br] at 1145 240
  \pinlabel $c_2$ [bl] at 1318 240
  \pinlabel $c_3$ [t] at 1225 96
  \pinlabel $c_4$ [r] at 1220 392
  \pinlabel $c_5$ [tr] at 1398 91
  \pinlabel $c_6$ [tl] at 1060 95
  \endlabellist

  \includegraphics[width=3.5in]{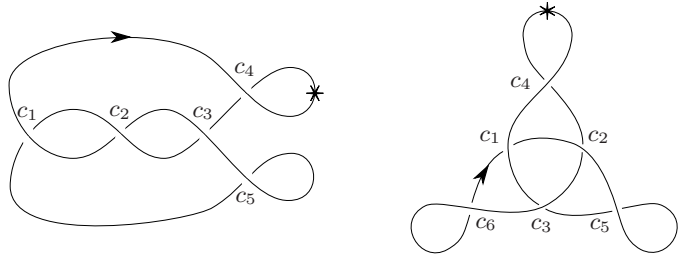}
  \caption{The left figure is a Legendrian right handed trefoil and
    the right is a Legendrian left handed trefoil with crossings
    labeled. The $*$ indicates the placement of the base point
    corresponding to $t$.}
  \label{fig:trefoilsEx}
\end{figure}

This result complements the recent work of Henry and Rutherford
\cite{Henry}. Henry and Rutherford show that counts of the
augmentations to any finite field, without restrictions on where the
augmentation sends $t$, are Legendrian knot invariants and that they
can be related to the ruling polynomials of the knot, thus showing
that the Chekanov-Eliashberg algebra determines the ruling
polynomial. Our result shows that if $\rho$ is even, one can restrict
the count of $\rho$-graded augmentations to augmentations which send
$t$ to $-1$, as there are not any which do not.

Theorem \ref{thm:main} tells us that if there exists an augmentation
to $\integers/2$, then there exists an augmentation to any field. In
\S \ref{sec:lift}, we will show that given an augmentation to
$\integers/2$ of the Chekanov-Eliashberg DGA, we can use constructions
similar to those in the proof of Theorem \ref{thm:main} to define an
augmentation to any ring. In particular:

\begin{thm}\label{thm:lift}
  Let $\Lambda$ be a Legendrian knot in $\reals^3$. Let
  $(\Az,\partial)$ be the Chekanov-Eliashberg DGA over $\integers/2$
  and let $(\A,\partial)$ be the DGA over $R=\integers[t,t^{-1}]$. If
  $\epsilon':\Az\to\integers/2$ is an augmentation of
  $(\Az,\partial)$, then one can find a lift of $\epsilon'$ to an
  augmentation $\epsilon:\A\to\integers$ of $(\A,\partial)$ such that
  $\epsilon(t)=-1$.
\end{thm}

In other words, we will define $\epsilon$ so that the following
diagram commutes:
\diag{(\A,\partial)\arr^{\epsilon}\ard_{t=1}&\integers\ard\\(\Az,\partial)\arr_{\epsilon'}&\integers/2}

This theorem tells us that given an augmentation to $\integers/2$ of
$(\Az,\partial)$, there exists an augmentation to any ring $S$ of
$(\A,\partial)$ which sends $t$ to $-1$.

\subsection{Outline of the article}
In \S \ref{sec:background} we recall background on Legendrian knots
and give definitions of the Chekanov-Eliashberg DGA, including sign
conventions for defining the algebra over $\integers[t,t^{-1}]$, and a
normal ruling. \S \ref{sec:augRuling} gives the proof that given an
augmentation one can define a normal ruling. \S \ref{sec:rulingAug}
finishes the proof of Theorem \ref{thm:main} by proving that given a
normal ruling one can define an augmentation. \S \ref{sec:rulingAug}
goes to prove Theorem \ref{thm:rhoOdd}, giving the augmentation
variety in the odd graded case. The paper concludes with the proof of
Theorem \ref{thm:lift} in \S \ref{sec:lift}.

\subsection{Acknowledgements}
The author thanks Lenhard Ng for introduction to the problem, for many
useful discussions, and for the contribution of the proof of Lemma
\ref{lem:oddNumBasepts}. The author also thanks Dan Rutherford for
helpful conversations. This work was partially supported by NSF grant
DMS-0846346.

\bigskip
\section{Background Material} \label{sec:background}
\subsection{Diagrams of Knots}
In this section, we will briefly review necessary ideas of Legendrian
knot theory. For further references on this subject, see
\cite{EtnyreLegendrianTrans}.

A {\bf contact structure} on a 3-manifold $M$ is a completely
nonintegrable 2-plane field $\xi$. Locally, a contact structure is the
kernel of a 1-form $\alpha$ which satisfies the non-degeneracy
condition
\[\alpha\wedge d\alpha\neq0\]
at every point in $M$. We will be concerned with the {\bf standard
  contact structure} on $\reals^3$, which is the completely
nonintegrable 2-plane field $\xi_0=\ker\alpha_0$, where
$\alpha_0=dz-ydx$. A {\bf Legendrian knot} is an embedding
$\Lambda:S^1\to\reals^3$ which is everywhere tangent to the contact
planes. A {\bf Legendrian isotopy} is an ambient isotopy of $\Lambda$
through Legendrian knots. We are interested in Legendrian isotopy
classes of Legendrian knots in $\reals^3$.

The classical invariants for Legendrian isotopy classes of knots are
the topological knot type, Thurston-Bennequin number, and rotation
number (see \cite{Bennequin}). The {\bf Thurston-Bennequin number}
measures the self-linking of a Legendrian knot $\Lambda$. If
$\Lambda'$ is a knot that is a push off of $\Lambda$ in a direction
tangent to the contact structure, then $tb(\Lambda)$ is the linking
number of $\Lambda$ and $\Lambda'$. The {\bf rotation number} $r$ of
an oriented Legendrian knot $\Lambda$ is the rotation of its tangent
vector field with respect to any global trivialization of $\xi_0$, for
example, $\{\partial_y,\partial_x+y\partial_z\}$. A natural question
is then whether these invariants with the topological knot type alone
classify Legendrian knots, in other words, whether all Legendrian
knots are ``Legendrian simple.'' Eliashberg and Fraser
\cite{EliashbergTrivial} show that Legendrian unknots are Legendrian
simple and Etnyre and Honda \cite{EtnyreTorus} show that Legendrian
torus and figure eight knots are as well.

Two particularly useful projections of Legendrian knots are the
Lagrangian projection and the front projection. The {\bf Lagrangian
  projection} is the map
\[\pi_\ell:(x,y,z)\mapsto(x,y).\]
The {\bf front projection} is the map
\[\pi_f:(x,y,z)\mapsto(x,z).\]
In general, we will call the Lagrangian projection (resp. front
projection) of a Legendrian knot a {\bf Lagrangian diagram}
(resp. {\bf front diagram}). Figure \ref{fig:trefoil} gives Lagrangian
(left) and front (right) projections of a Legendrian version of a
right handed trefoil.

\begin{figure}
  \labellist
  \small\hair 2pt 
  \pinlabel $c_1$ [t] at 50 195 
  \pinlabel $c_2$ [t] at 255 195 
  \pinlabel $c_3$ [t] at 447 195 
  \pinlabel $q_1$ [b] at 545 326 
  \pinlabel $q_2$ [t] at 545 92
  \endlabellist
  \includegraphics[scale=.2]{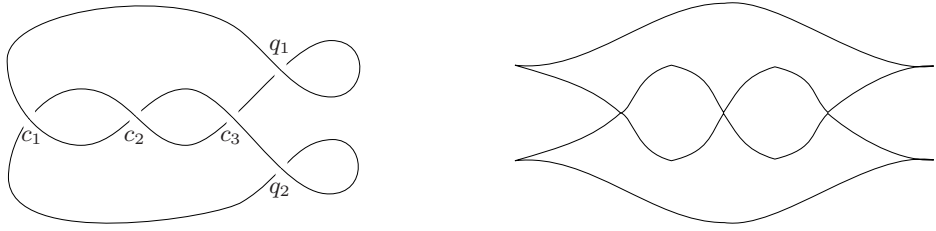}
  \caption{The left figure gives a Lagrangian projection of a
    Legendrian right handed trefoil with crossings labeled and the
    right figure gives a front projection.}
  \label{fig:trefoil}
\end{figure}

Note that one can recover the $y$ coordinate of a knot from the slope
of the front diagram (see \cite{EtnyreLegendrianTrans}):
\[y=\frac{dz}{dx}.\] This implies that lines tangent to a front
diagram of a Legendrian knot are never vertical. Front diagrams
instead have semicubical cusps. It also implies that at a double point
the strand with the smaller (more negative) slope has a smaller $y$
coordinate and so passes in front of the strand with larger (more
positive) slope. For a front diagram of an oriented Legendrian knot,
the rotation number is half of the difference between the number of
downward-pointing cusps and the number of upward-pointing cusps.

In particular, we will find that front diagrams in plat position will
be easier to manipulate. A front diagram is in {\bf plat position} if
all of the left cusps have the same $x$ coordinate, all of the right
cusps have the same $x$ coordinate, and there do not exist crossings
in the diagram which have the same $x$ coordinate. One can use
Legendrian versions of the Reidemeister II moves and planar isotopy to
put any front diagram into plat position. The diagram of the trefoil
given in Figure \ref{fig:trefoil} is an example of a diagram in plat
position.

\subsection{Definition of the DGA and augmentations}
This section contains a brief overview of the differential graded
algebra presented by Etnyre, Ng, Sabloff in \cite{EtnyreInvariants}
which lifts the Chekanov-Eliashberg differential graded algebra over
$\integers/2$ in \cite{Chekanov} to a DGA over $\integers[t,t^{-1}]$.

Given a front diagram of an oriented Legendrian knot $\Lambda$ in plat
position in $\reals^3$ with the standard contact structure, Ng's
resolution process \cite{NgComputable} gives a Lagrangian diagram for
a knot Legendrian isotopic to $\Lambda$ by smoothing left cusps,
replacing right cusps with a loop, and resolving crossings so that the
over crossing strand has smaller (more negative) slope.

\begin{note}
  Label the crossings of the Lagrangian resolution of a front diagram
  of $\Lambda$ in plat position by $\{c_1,\ldots,c_n,q_1,\ldots,q_m\}$
  with $q_1,\ldots,q_m$ the crossings from resolving the right cusps
  labeled from the top to the bottom and $c_1,\ldots,c_n$ the
  remaining crossings labeled from left to right (see Figure
  \ref{fig:trefoilDGA}). Label each quadrant around a crossing as shown
  in Figure \ref{fig:reebSigns}. We will refer to these labels as the
  {\bf Reeb signs} and will call a quadrant at a crossing {\bf
    positive} or {\bf negative} depending on its Reeb sign.
\end{note}

\begin{figure}
  \labellist
  \small\hair 2pt 
  \pinlabel $-$ [b] at 44 53 
  \pinlabel $-$ [t] at 44 39 
  \pinlabel $+$ [l] at 50 45 
  \pinlabel $+$ [r] at 35 45
  \endlabellist
  \includegraphics[scale=.9]{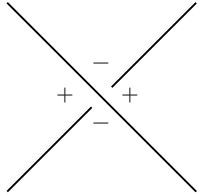}
  \caption{A labeling of the Reeb signs of the quadrants around a
    crossing.}
  \label{fig:reebSigns}
\end{figure}

\begin{defn}
  Let $\Lambda$ be an oriented Legendrian knot in plat position
  decorated with $*$ for the base point. The algebra
  $\calA_R(c_1,\ldots,c_n,q_1,\ldots,q_m)$ is the noncommutative
  graded free associative unital algebra over $R=\integers[t,t^{-1}]$
  generated (as an algebra) by $\{c_1,\ldots,c_n,q_1,\ldots,q_m\}$. We
  will sometimes shorten this to $\calA_R$.

  The grading for $t$ is defined to be $-2r(\Lambda)$. To give $c_i$ a
  grading, we first must specify a capping path $\gamma_{c_i}$. The
  {\bf capping path} $\gamma_{c_i}$ is the unique path in $\Lambda$
  which begins at the under crossing of $c_i$, ends at the over
  crossing of $c_i$, and does not go through the base point $*$ (note
  that this may mean the capping path has the opposite orientation of
  the knot), as seen in Figure \ref{fig:cappingPath}.
\end{defn}

\begin{figure}
  \labellist
  \small
  \pinlabel $+$ [r] at 60 34 
  \pinlabel $+$ [l] at 68 34
  \pinlabel $-$ [b] at 64 38 
  \pinlabel $-$ [t] at 64 30
  \endlabellist

  \includegraphics[width=1.5in]{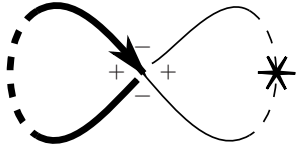}
  \caption{The choice of capping path for a crossing, where the
    capping path is denoted by a heavy line and the arrows give the
    orientation of the capping path. The signs are the Reeb signs.}
  \label{fig:cappingPath}
\end{figure}

Define the rotation number $r(\gamma_{c_i})$ to be the fractional
number of counterclockwise revolutions made by the tangent vector to
$\gamma_{c_i}$ as we follow the path. One can perturb the diagram of
$\Lambda$ so that all crossings are orthogonal and thus
$r(\gamma_{c_i})$ is an odd multiple of $1/4$. Define the grading on
$c_i$ by
\[\lvert c_i\rvert=-2r(\gamma_{c_i})-\frac12.\]
(Note that by setting $t=1$ we recover Chekanov's grading from
\cite{Chekanov}, though we then need to consider the grading mod
$2r(\Lambda)$.)

Since we are working with front projections of knots in plat position,
we can assign the gradings mod $2r(\Lambda)$ of crossings at right
cusps: $\lvert q_k\rvert=1$. Let $C(\Lambda)$ be the set of points
on $\Lambda$ corresponding to cusps of the front projection of
$\Lambda$. A {\bf Maslov potential function} is a locally constant
function
\[\mu:\Lambda\backslash C(\Lambda)\to\integers/2r(\Lambda)\]
such that for two strands meeting at a cusp (either left or right),
the upper strand has Maslov potential one higher than the lower
strand. Such a function is well-defined up to a constant. Near a
crossing $c_k$, let $\alpha_k$ be the strand in the front diagram with
more negative slope and let $\beta_k$ be the strand with more positive
slope. The grading defined earlier now becomes
\[\lvert c_k\rvert\equiv\mu(\alpha_k)-\mu(\beta_k)\mod2r(\Lambda).\]

Label a point on the diagram $*$. This will be the base point
corresponding to $t$. In \S \ref{sec:basepoints} we will discuss the
case when we have multiple base points. We define the differential
$\partial$ on $\calA_R(c_1,\ldots,c_n,q_1,\ldots,q_m)$ by
appropriately counting embedded disks in the Lagrangian resolution of
the front projection of $\Lambda$ in plat position. (Note that, in
general, one would need to look for immersed disks, but since
$\Lambda$ is in plat position, we need only look for embedded disks.)

Given a generator $a$ and an ordered set of generators
$\{b_1,\ldots,b_k\}$, let $\Delta(a;\{b_1,\ldots,b_k\})$ be the set of
orientation-preserving embeddings
\[f:D^2\to\reals^2\] (up to smooth reparametrization) that map
$\partial D^2$ to the Lagrangian resolution of $\pi_f(\Lambda)$, such
that
\begin{enumerate}
\item the restriction of $f$ to $\partial D^2$ is an embedding except
  at $a,b_1,\ldots,b_k$,
\item $a,b_1,\ldots,b_k$ are encountered in counter-clockwise order
  along $f(\partial D^2)$,
\item near $a,b_1,\ldots,b_k$, $f(D^2)$ covers exactly one quadrant,
  specifically, a quadrant with positive Reeb sign near $a$ and a
  quadrant with negative Reeb sign near $b_i$ for $1\leq i\leq k$.
\end{enumerate}

We can assign a word in $\calA$ to each embedded disk by starting with
the first corner after the one covering the $+$ quadrant and listing
the crossing labels of all negative corners as encountered while
following the boundary of the immersed polygon counter-clockwise. We
associate a sign to each immersed disk by associating an {\bf
  orientation sign} $\epsilon_{Q,a}$ to each quadrant $Q$ in the
neighborhood of a crossing $a$, determined by Figure
\ref{fig:orientation}, and defining the sign of a disk $f(D^2)$, the
product of the orientation signs over all the corners of the disk,
denoted $\epsilon(f(D^2))$. Since we are working with a diagram in
plat position, in practice, we can define $\epsilon(a;b_1\cdots b_k)$
to be the sign of the unique disk with positive corner at $a$ (with
respect to Reeb signs) and negative corners at $b_1,\ldots,b_k$, the
product of the orientation signs over all corners of the disk. Note
that our convention for assigning orientation signs differs from
\cite{EtnyreInvariants}. At any crossing $c$ where our convention
differs from that in \cite{EtnyreInvariants}, one can recover the
convention in \cite{EtnyreInvariants} by sending $c$ to $-c$.

\begin{figure}
  \labellist
  \small \hair 3pt
  \pinlabel $+$ [b] at 56 57
  \pinlabel $+$ [t] at 56 57
  \pinlabel $-$ [l] at 56 57
  \pinlabel $-$ [r] at 56 57

  \pinlabel $+$ [br] at 248 55
  \pinlabel $+$ [tl] at 248 55
  \pinlabel $-$ [bl] at 248 55
  \pinlabel $-$ [tr] at 248 55

  \pinlabel $+$ [l] at 447 54
  \pinlabel $+$ [r] at 447 54
  \pinlabel $-$ [t] at 447 54
  \pinlabel $-$ [b] at 447 54
  \endlabellist

  \includegraphics[width=4in]{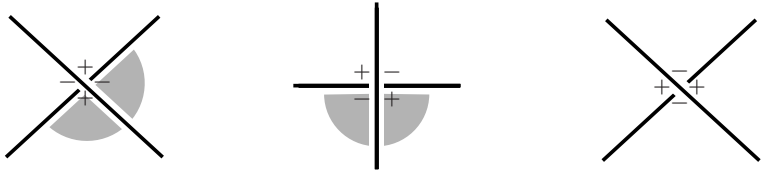}
  \caption{The left two diagrams are positive crossings while the
    right one is a negative crossing. The signs in the figure are Reeb
    signs. The orientation signs are $-1$ for the shaded quadrants and
    $+1$ everywhere else. The middle diagram gives the orientation
    assignment for a positive crossing in a dip, which will be
    discussed in \S \ref{sec:dips}.}
  \label{fig:orientation}
\end{figure}

Define $n_*(a;b_1,\ldots,b_k)$ to be the signed count of the number of
times one encounters the base point $*$ while following $f(\partial
D^2)$ in the counter-clockwise direction, where the sign is determined
by whether one encounters the base point while following the
orientation of the knot or going against the orientation of the knot.

\begin{defn} 
  The algebra $\calA_R$ is a differential graded algebra
  (DGA) whose differential $\partial$ is defined as follows:
  \[\partial a=\sum_{b_1,\ldots,b_k}\epsilon(a;b_1\cdots
  b_k)t^{n_*(a;b_1,\ldots,b_k)}b_1\cdots b_k.\] Extend $\partial$ to
  $\calA_R$ via $\partial(\integers[t,t^{-1}])=0$ and the signed
  Leibniz rule:
  \[\partial(vw)=(\partial v)w+(-1)^{\lvert v\rvert}v(\partial w).\]
\end{defn}

From Theorem 3.7 in \cite{EtnyreInvariants}, the differential
$\partial$ has degree $-1$ and satisfies $\partial^2=0$.

\begin{figure}
  \labellist
  \small
  \pinlabel $c_1$ [b] at 43 220
  \pinlabel $c_2$ [b] at 256 222
  \pinlabel $c_3$ [b] at 449 222
  \pinlabel $q_1$ [b] at 544 318
  \pinlabel $q_2$ [b] at 544 122
  \endlabellist
  \includegraphics[width=2in]{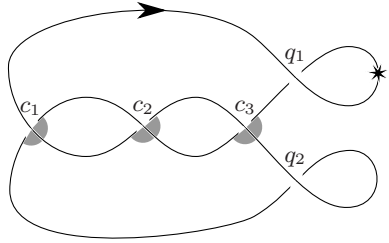}
  \caption{The Lagrangian resolution of the front diagram of the right
    trefoil in plat position. Crossings are labeled and $*$ indicates
    the base point corresponding to $t$. The shaded regions are
    quadrants with orientation sign $-1$. All other quadrants have
    orientation sign $+1$.}
  \label{fig:trefoilDGA}
\end{figure}

For example, the right handed trefoil depicted in Figure \ref{fig:trefoilDGA}
with $r=0$ and $tb=1$ has $\lvert c_i\rvert=0$ and $\lvert
q_i\rvert=1$. We have $\calA_R=\calA_R(c_1,c_2,c_3,q_1,q_2)$ with
differential
\begin{align*}
  \partial c_1&=\partial c_2=\partial c_3=0\\
  \partial q_1&=t+c_1+c_3+c_1c_2c_3\\
  \partial q_2&=1-c_1-c_3-c_3c_2c_1.
\end{align*}

\begin{defn} 
  A graded chain isomorphism
  \[\phi:\A(a_1,\ldots,a_n)\to\A(b_1,\ldots,b_n)\]
  is {\bf elementary} if there exists $j\in\{1,\ldots,n\}$ such that
  \[\phi(a_i)=\begin{cases}
    b_i&i\neq j\\
    ub_j+v&v\in\A(b_1,\ldots,b_{j-1},b_{j+1},\ldots,b_n), u\text{ a
      unit in }R,i=j.
  \end{cases}\] A composition of elementary isomorphisms is called
  {\bf tame}.
\end{defn}

\begin{defn} 
  Define the algebra $\calE_i=\calA(e_1^i,e_2^i)$ by
  setting $\lvert e_1^i\rvert=i-1$, $\lvert e_2^i\rvert=i$, $\partial
  e_2^i=e_1^i$, and $\partial e_1^i=0$.
\end{defn}
 
This algebra models the second Reidemeister move, which produces two
new crossings.

\begin{defn} Given a DGA $(\calA(a_1,\ldots,a_n),\partial)$, the {\bf
    degree $i$ stabilization} of $(\calA(a_1,\ldots,a_n),\partial)$ is
  defined to be $\calA(a_1,\ldots,a_n,e_1^i,e_2^i)$. The grading and
  the differential are inherited from $\calA$ and $\calE_i$. Two DGA's
  $(\calA,\partial)$ and $(\calA',\partial')$ are {\bf stable tame
    isomorphic} if there exist two sequences of stabilizations
  $S_{i_1},\ldots,S_{i_n}$ and $S_{j_1},\ldots,S_{j_m}$ and a tame
  isomorphism
  \[\phi:S_{i_n}(\cdots(S_{i_1}(\calA))\cdots)\to
  S_{j_m}(\cdots(S_{j_1}(\calA'))\cdots),\] which is also a chain map.
\end{defn}

In fact, the stable tame isomorphism class of the DGA is invariant
under Legendrian isotopy. Chekanov proved this result over
$\integers/2$ in \cite{Chekanov} and Etnyre, Ng, and Sabloff proved
this result over $\integers[t,t^{-1}]$ in \cite{EtnyreInvariants}.

\vspace{.1in} Now that we have the DGA associated with the projection
of $\Lambda$, we can discuss the augmentations.

\begin{defn}
  Let $F$ be a field. An {\bf augmentation} of $(\A,\partial)$ to $F$
  is an algebra map $\epsilon:\A\to F$ such that
  $\epsilon\circ\partial=0$ and $\epsilon(1)=1$. If
  $\rho\vert2r(\Lambda)$ and $\epsilon$ is supported on generators of
  degree divisible by $\rho$, then $\epsilon$ is {\bf
    $\rho$-graded}. In particular, if $\rho=0$, we say it is {\bf
    graded} and if $\rho=1$, we say it is {\bf ungraded.} We call a
  generator $a$ {\bf augmented} if $\epsilon(a)\neq0$.
\end{defn}

For example, if we recall the DGA over $\integers[t,t^{-1}]$ for the
right handed trefoil, then we can classify the augmentations to any
field $F$ as follows: Let $\epsilon:\calA_R\to F$ be an
augmentation. Then $\epsilon(t)=-1$ and
\begin{itemize}
\item if $\epsilon(c_1)=0$, then $\epsilon(c_3)=1$ and
  $\epsilon(c_2)\in F$
\item if $\epsilon(c_3)=0$, then $\epsilon(c_1)=1$ and
  $\epsilon(c_2)\in F$
\item if $\epsilon(c_1),\epsilon(c_3)\neq0$, then
  \[\epsilon(c_2)=(1-\epsilon(c_1)-\epsilon(c_3))(\epsilon(c_1))^{-1}(\epsilon(c_3))^{-1}.\]
\end{itemize}
Note that if $F$ is a finite field, as in \cite{Henry}, and $\lvert
F\rvert$ is the number of elements in $F$, then we see that there are
$\lvert F\rvert$ augmentations of the first type, $\lvert F\rvert$
augmentations of the second type, and $\lvert F^*\rvert^2$
augmentations of the third type, where $F^*=F\backslash\{0\}$. In
fact,
\begin{equation}\label{eqn:augHenry}
  \{(\epsilon(c_1),\epsilon(c_2),\epsilon(c_3),\epsilon(q_1),\epsilon(q_2),\epsilon(t)):\epsilon\text{
    an augmentation to }F\}=F\coprod F\coprod(F^*)^2.
\end{equation}
In \cite{Henry}, this is called the augmentation variety of
$(\calA(\Lambda),\partial)$. Comparing this with possible rulings of
the trefoil, definition given in \S\ref{sec:rulings}, one sees that
\eqref{eqn:augHenry} coincides with Theorem 3.4 of \cite{Henry}.

For example, the following are examples of graded augmentations to
$\reals$.
\[\begin{array}{c|cccccc}
  &c_1&c_2&c_3&q_1&q_2&t\\
  \hline \epsilon_1&1&\frac12&0&0&0&-1\\
  \epsilon_2&0&\frac12&1&0&0&-1\\
  \epsilon_3&2&\frac34&-\frac25&0&0&-1\\
  \epsilon_4&-\frac25&\frac34&2&0&0&-1\\
  \epsilon_5&\frac12&0&\frac12&0&0&-1
\end{array}\]

Note that any augmentation of a stabilization $S(\calA)$ restricts to
an augmentation of the smaller algebra $\calA$ and any augmentation of
the algebra $\calA$ extends to an augmentation of the stabilization
$S(\calA)$ where the augmentation sends $e^i_1$ to $0$ and $e^i_2$ to
an arbitrary element of $F$ if $\rho\vert i$ and $0$ otherwise.

\subsection{Rulings}\label{sec:rulings}
This paper will show that there is a way to construct an augmentation
from a normal ruling and a normal ruling from an augmentation. 

\begin{defn}
  Consider a front diagram in plat position of a Legendrian knot
  $\Lambda$. A {\bf ruling} of this diagram consists of a one-to-one
  correspondence between the set of left cusps and the set of right
  cusps where, for each pair of corresponding cusps, two paths in the
  front diagram join them. These {\bf ruling paths} must satisfy the
  following:
  \begin{enumerate}
  \item Any two paths in the ruling only meet at crossings or
    cusps;\label{cond:1}
  \item The interiors of the two paths joining corresponding cusps are
    disjoint. Thus each pair of paths bound a topological
    disk. \label{cond:2}
  \end{enumerate}
\end{defn}

The first condition tells us the ruling paths never overlap at more
than a finite number of points. The second condition tells us that
there are disks similar to those in the differential $\partial$, but
possibly with ``obtuse'' corners. As noted in \cite{FuchsAug},
these imply that the ruling paths cover the front diagram and the
$x$-coordinate of each path in the ruling is monotonic.

Near a crossing, the two ruling paths which intersect at the crossing
are called {\bf crossing paths}. The two paths paired with the
crossing paths are called {\bf companion paths}.

Given a ruling, at any crossing, we either have that the crossing
paths pass through each other, or one path lies entirely above (has
$z$-coordinate strictly greater than) the other. In the latter case,
we say the ruling is {\bf switched} at the crossing. If all of the
switched crossings in the ruling are of the form (a), (b), or (c), as
seen in Figure \ref{fig:config} then we say the ruling is {\bf
  normal}. Thus, the possible configurations near a crossing in a
normal ruling are shown in Figure \ref{fig:config}.

\begin{figure}
  \labellist
  \small\hair 2pt
  \pinlabel $(a)$ [t] at 82 267
  \pinlabel $(b)$ [t] at 339 267
  \pinlabel $(c)$ [t] at 593 267
  \pinlabel $(d)$ [t] at 82 -20
  \pinlabel $(e)$ [t] at 339 -20
  \pinlabel $(f)$ [t] at 593 -20
  \endlabellist
  \includegraphics[width=4.5in]{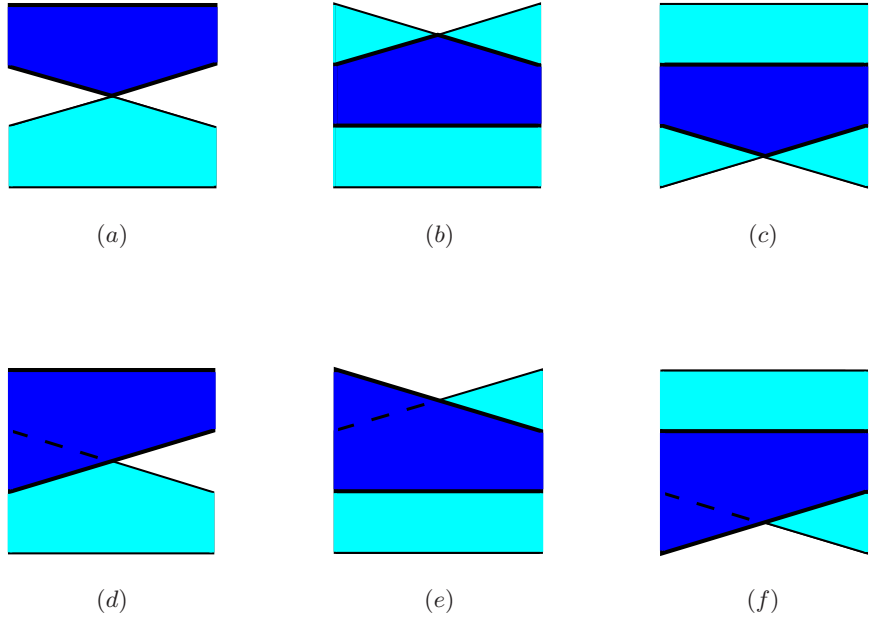}
  \vspace{.25in}
  \caption{By including vertical reflections of (d), (e), and (f),
    these are all possible configurations of crossings appearing in a
    normal ruling. The top row contains all possible configurations
    for switched crossings in a normal ruling.}
  \label{fig:config}
\end{figure}
\vspace{.5in}

\noindent If all of the switched crossings have grading divisible by
$\rho$ for some $\rho$ such that $\rho\vert 2r(\Lambda)$, then we say
the ruling is {\bf $\rho$-graded}. In particular, if $\rho=0$, then we
say the ruling is {\bf graded} and if $\rho=1$, then we say the ruling
is {\bf ungraded}.

For example, if $F=\integers/2\integers$, the trefoil has three graded
normal rulings as seen in Figure \ref{fig:trefoilRulings}.

\begin{figure}
  \includegraphics[width=5.5in]{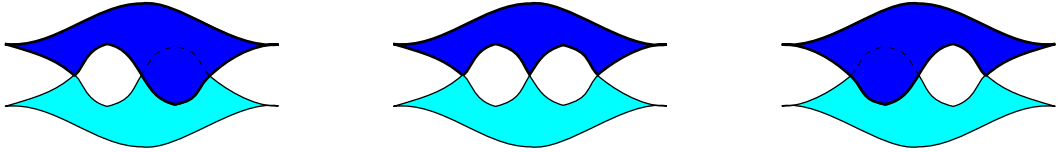}
  \caption{The graded normal rulings of the right handed trefoil.}
  \label{fig:trefoilRulings}
\end{figure}

In \cite{ChekanovFronts}, Chekanov showed that the number of
$\rho$-graded normal rulings is invariant under Legendrian isotopy.

\subsection{Dips}\label{sec:dips}
We will construct a normal ruling of the diagram by using the
augmentation to construct an augmentation $\epsilon$ of the dipped
diagram satisfying Property (R), as called in
\cite{SabloffAug}. However, the notation in the following section will
be necessary to write down Property (R).

Given a Legendrian knot $\Lambda$ in plat position, we construct a
{\bf dip} between two crossings by a sequence of Reidemeister II
moves, as seen in Figure \ref{fig:dips} in the front projection and
Lagrangian projection. In the front projection, it is clear that the
diagram with the dip is isotopic to the original diagram. To construct
a dip, number the $2m$ strands from bottom to top. Using a type II
Reidemeister move, push strand $2$ over strand $1$, then strand $3$
over strand $1$, then strand $3$ over strand $2$, and so on. So that
strand $k$ is pushed over strand $\ell$ in lexicographic order. If
strand $k$ crosses strand $\ell$ after strand $i$ crosses strand $j$,
we write $(i,j)<(k,\ell)$.

The {\bf dipped diagram} involves introducing a dip between each
crossing in the plat position diagram and between the left,
respectively right, cusps and the first, respectively last, crossing
(see Figure \ref{fig:trefoilEx}). Each Reidemeister II move introduces
two new variables. For the dip immediately after crossing $c_k$, we
will use $a^k_{rs}$ and $b^k_{rs}$ to denote the new crossings
introduced when strand $r$ is passed over strand $s$ ($r>s$), with
$b^k_{rs}$ being the leftmost and $a^k_{rs}$ being the rightmost new
crossing (see Figure \ref{fig:dips}). We will say the $b^k_{rs}$
generators belong to the {\bf $b^k$-lattice} and the $a^k_{rs}$ belong
to the {\bf $a^k$-lattice}. Thus we will have $a^k/b^k$-lattices for
$0\leq k\leq n$. While dipped diagrams have many more crossings than
the original knot diagram, the differential $\partial$ on $\calA_R$ is
generally much simpler. We note that if $\mu$ is a Maslov potential
function on the front diagram, then
\[\lvert b^k_{rs}\rvert=\mu(r)-\mu(s).\]
Since the differential $\partial$ lowers degree by one,
\[\lvert a^k_{rs}\rvert=\lvert b^k_{rs}\rvert-1.\]

% picture of dip in front and Lagrangian projections
\begin{figure}
  \labellist
  \small\hair 1pt
  \pinlabel $1$ [r] at 320 99
  \pinlabel $2$ [r] at 320 132
  \pinlabel $3$ [r] at 320 164
  \pinlabel $4$ [r] at 320 198

  \pinlabel $b_{41}$ [tr] at 425 99
  \pinlabel $b_{42}$ [tr] at 425 67
  \pinlabel $b_{43}$ [tr] at 425 35
  \pinlabel $b_{31}$ [tr] at 393 99
  \pinlabel $b_{32}$ [tr] at 393 67
  \pinlabel $b_{21}$ [tr] at 362 99

  \pinlabel $a_{41}$ [tl] at 472 99
  \pinlabel $a_{42}$ [tl] at 472 67
  \pinlabel $a_{43}$ [tl] at 472 35
  \pinlabel $a_{31}$ [tl] at 505 99
  \pinlabel $a_{32}$ [tl] at 505 67
  \pinlabel $a_{21}$ [tl] at 537 99

  \endlabellist

  \includegraphics[width=4in]{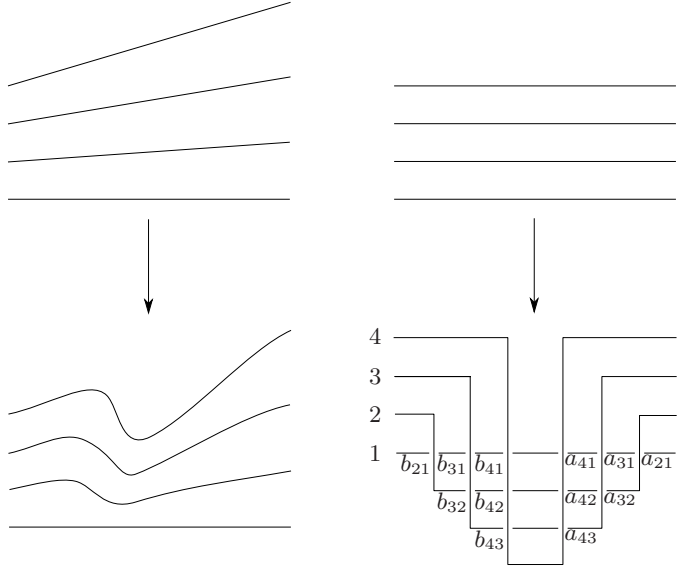}
  \caption{ The left diagram gives the modification of the front
    projection when creating a dip, while the right diagram gives the
    modification of the Lagrangian projection. In the Lagrangian
    projection, the $b^k$-lattice is made up of the crossings on the
    left and the $a^k$-lattice is made up of the crossings on the
    right. The crossings in the $b^k$-lattice are labeled down and to
    the left, while the crossings in the $a^k$-lattice to the right, with
    $k$'s suppressed.}
  \label{fig:dips}
\end{figure}

Orientation sign assignments are given in Figure
\ref{fig:orientation}.  We can reduce possible disks, and thus
possible terms in the differential, further in certain cases. As the
disks in the computation of $\Az$ are the same disks in the
computation of $\A$, we have the following lemma from
\cite{SabloffAug}.

\begin{lem}[\cite{SabloffAug} Lemma 3.1]\label{lem:3.1equiv}
  If $a$ and $b$ are the new crossings created by a type II move
  during the creation of a dip and $y$ is any other crossing, then $a$
  appears at most once in any term of $\partial y$, and if $a$ appears
  in any term of $\partial y$, then $b$ does not.
\end{lem}

This follows from considering the disks which have a negative corner
at $a$ as seen in Figure \ref{fig:disksA}.

\begin{figure}
  \labellist
  \small\hair 2pt
  \pinlabel $b$ [bl] at 100 64
  \pinlabel $a$ [br] at 147 64

  \pinlabel $b$ [bl] at 528 64
  \pinlabel $a$ [br] at 575 64

  \tiny
  \pinlabel $-$ [bl] at 147 64
  \pinlabel $+$ [tl] at 147 97
  \pinlabel $-$ [tr] at 212 97
  \pinlabel $-$ [bl] at 575 65
  \endlabellist
  \includegraphics[width=5in]{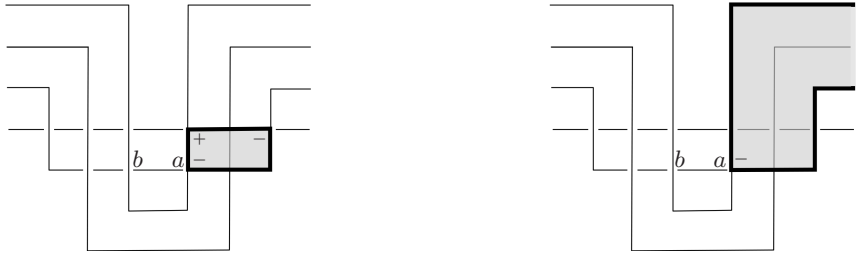}
  \caption{Possible disks contributing to $\partial$ with a negative
    corner at $a$.}
  \label{fig:disksA}
\end{figure}

Through consideration of the dipped diagram, we see
\begin{itemize}
\item the differential of crossings in the $b^k$-lattice involve at
  most
  \begin{itemize}
  \item $c_k$,
  \item base points (we will discuss the case when we have more than
    one in the next section),
  \item crossings in the $a^{k-1}$-lattice,
  \item crossings in the $a^k$-lattice,
  \end{itemize}
\item the differential of crossings in the $a^k$-lattice only involve
  \begin{itemize}
  \item base points,
  \item crossings in the $b^k$-lattice,
  \end{itemize}
\item the differential of $c_k$ is $0$
\end{itemize}
for all $1\leq k\leq n$. This greatly reduces the types of totally
augmented disks for which to look to compute whether we have an
augmentation, where a {\bf totally augmented disk} is a disk which
contributes to the differential, all of whose negative corners are
augmented.

\bigno \note $a^k_{\{r,s\}}=a^k_{\max(r,s),\min(r,s)}$

\subsection{Augmentations before and after a base point  move} 
\label{sec:basepoints} 
As we create dips, we will find that the signs are simpler if, in
certain cases, we add in a few extra base points. In
\cite{NgSatellites}, Ng and Rutherford give the DGA isomorphisms
induced by adding a base point and by moving one base point around a
knot. First, we need to extend our definition of the DGA over
$\integers[t,t^{-1}]$ to a DGA over
$\integers[t_1^{\pm1},\ldots,t_s^{\pm1}]$, which we will call
$\calA(\Lambda,*)$. To this end, label $s$ points on the Lagrangian
resolution of the front diagram of $\Lambda$ by the base points
$*_1,\ldots,*_s$ respectively associated to $t_1,\ldots,t_s$.

\begin{defn}
  The algebra $\calA$ is a DGA whose grading is defined analogously to
  the case when there is only one base point: We define $\lvert
  t_1\rvert=-2r(\Lambda)$ and $\lvert t_i\rvert=0$ for $1<i\leq
  s$. Given a crossing $c$, let $\gamma_c$ be the unique path
  following the under strand of $c$ to the over strand of $c$ while
  avoiding $*_1$ and define $\lvert
  c\rvert=-2r(\gamma_c)-\frac12$. The differential $\partial$ is
  defined as follows:
  \[\partial a=\sum_{b_1,\ldots,b_k}\epsilon(a;b_1\cdots
  b_k)t_1^{n_{*_1}(a;b_1,\ldots,b_k)}\cdots
  t_s^{n_{*_s}(a;b_1,\ldots,b_k)}b_1\cdots b_k.\] Extend $\partial$ to
  $\calA$ via $\partial(\integers[t,t^{-1}])=0$ and the signed Leibniz
  rule:
  \[\partial(vw)=(\partial v)w+(-1)^{\lvert v\rvert}v(\partial w).\]
\end{defn}

\begin{thm}[\cite{NgSatellites} Thm 2.19]\label{thm:2.19equiv}
  The map $\partial:\calA(\Lambda,*)\to\calA(\Lambda,*)$ lowers degree
  by 1 and is a differential: $\partial^2=0$. Up to stable tame
  isomorphism, the differential graded algebra
  $(\calA(\Lambda,*),\partial)$ is an invariant of $\Lambda$ under
  Legendrian isotopy (and choice of base point).
\end{thm}

\begin{thm}[\cite{NgSatellites} Thm 2.20]\label{thm:2.20equiv}
  Let $*_1,\ldots,*_k$ and $*'_1,\ldots,*'_k$ denote two collections
  of base points on the Lagrangian resolution of the front diagram of
  a Legendrian knot $\Lambda$, each of which is cyclically ordered
  along $\Lambda$. Let $(\calA(\Lambda,*_1,\ldots,*_k),\partial)$ and
  $\calA(\Lambda,*'_1,\ldots,*'_k),\partial')$ denote the
  corresponding multi-pointed DGAs. Then there is a DGA isomorphism
  $\Psi:(\calA(\Lambda,*_1,\ldots,*_k),\partial)\to(\calA(\Lambda,*'_1,\ldots,*'_k),\partial')$
  such that $\Psi(t_i)=t_i$ for all $i$.
\end{thm}

In the proof of this theorem, the isomorphism $\Psi$ is defined so
that $\Psi(c_j)=c_j$ if no base point is pushed over or under the
crossing $c_j$. If, however, the base point $*_i$ is pushed over
crossing $c_j$, then $\Psi(c_j)=t_i^{\pm1}c_j$, the sign depending on
whether the base point is pushed along the knot in the direction of
the orientation or against the orientation of the knot. If the base
point $*_i$ is pushed under the crossing $c_j$, then
$\Psi(c_j)=c_jt_i^{\pm1}$, again, the sign depending on the
orientation of the knot.

\begin{thm}[\cite{NgSatellites} Thm 2.21]\label{thm:2.21equiv}
  Let $*_1,\ldots,*_k$ be a cyclically ordered collection of base
  points along $\Lambda$, and let $*$ be a single base point on
  $\Lambda$. Then there is a DGA homomorphism
  $\phi:(\calA(\Lambda,*),\partial)\to(\calA(\Lambda,*_1,\ldots,*_k),\partial)$
  such that $\phi\circ\partial=\partial\circ\phi$ and
  $\phi(t)=t_1\cdots t_k$.
\end{thm}

Thus, we can assume there is one base point on each of the right
cusps.  Also, this shows us that if $\epsilon'$ is an augmentation on
the diagram after moving the base point $*_i$ over the crossing $c_j$,
then $\epsilon=\epsilon'\Psi$ is an augmentation on the diagram before
moving the base point.

\begin{rmk}\label{rmk:basepointMoving}
  In summary, if $\epsilon(t_i)=-1$, then moving the base point $*_i$
  over or under a crossing only changes the augmentation by changing
  the sign of the augmentation on that crossing, no matter the
  orientation of the strand.
\end{rmk}

Note that these theorems tell us that if $t$ is the variable
associated to the original base point $*$, and $t_1,\ldots,t_s$ are
the variables associated to the base points $*_1,\ldots,*_s$ in the
new diagram, $\epsilon'$ is an augmentation on the original diagram,
and $\epsilon$ is augmentation on the new diagram resulting from
Theorem \ref{thm:2.21equiv}, then
\[\epsilon'(t)=\epsilon(t_1\cdots t_s)=\prod_{i=1}^s\epsilon(t_i).\]

\subsection{Augmentations before and after type II moves}
To understand how augmentations before the addition of a dip relate to
augmentations after, we need to consider the stable DGA isomorphism
induced by a type II move. Suppose $(\calA_Z',\partial')$ is the DGA
over $Z$ for a knot diagram before a type II move and that
$(\calA_Z,\partial)$ is the DGA over $Z$ afterward. So

\begin{align*}
  \calA_Z&=\calA_Z(a,b,a_1,\ldots,a_r,b_1,\ldots,b_s;\partial)\\
  \calA'_Z&=\calA_Z(a_1,\ldots,a_r,b_1,\ldots,b_s;\partial'),
\end{align*}
where $Z=\integers[t_1,t_1^{-1},\ldots,t_q,t_q^{-1}]$.  Suppose that
the other crossings are ordered by height:
\[h(b_s)\geq\dots\geq h(b_1)\geq h(b)>h(a)\geq h(a_1)\geq\dots\geq
h(a_r).\]

It is possible to construct a dip in the plat diagram so that this
ordering takes the following form: Suppose strand $k$ is pushed over
strand $\ell$. Each $a_j$ either lies to the left of the dip or $a_j=a_{mn}$ or
$b_{mn}$ with $m-n\leq k-\ell$. Similarly, $b_j$ either lies to the
right of the dip or $b_j=a_{mn}$ or $b_{mn}$ with $m-n>k-\ell$.

\begin{figure}
  \labellist
  \small\hair 2pt
  \pinlabel $b$ [b] at 92 173
  \pinlabel $a$ [b] at 264 173
  \endlabellist
  \includegraphics[width=3.5in]{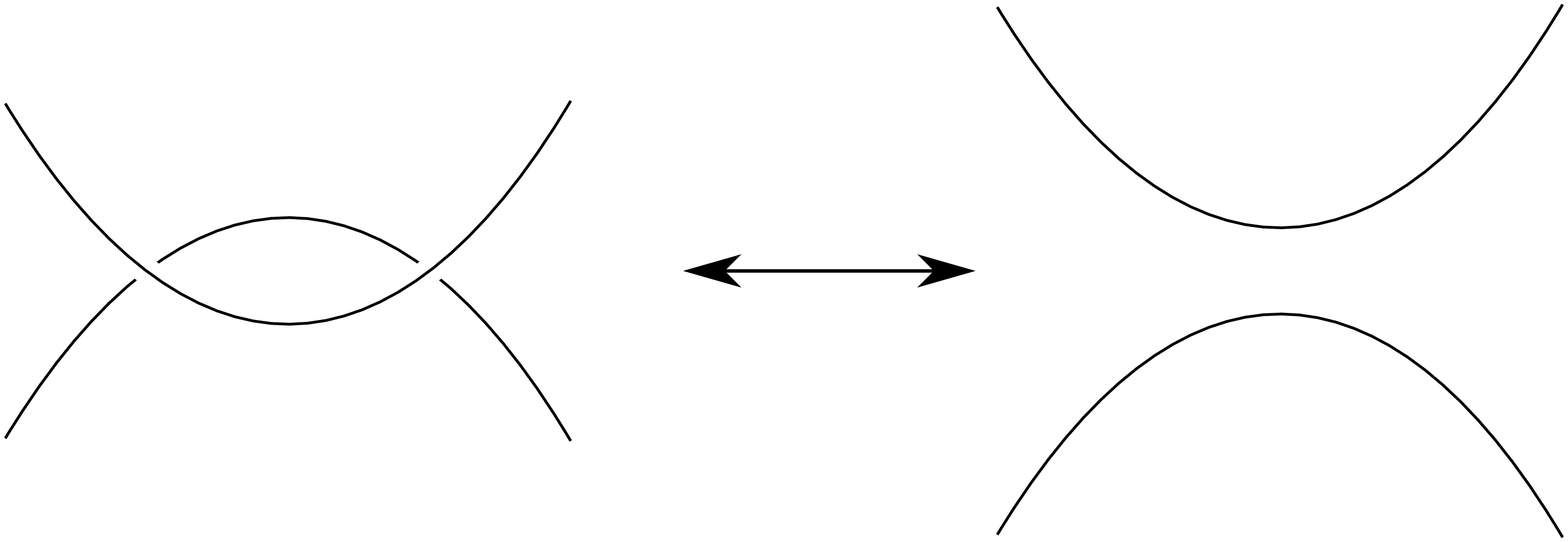}
  \caption{A type II Reidemeister move with crossings $a$ and $b$.}
    % and
    % quadrant orientation signs labeled as in Figure \ref{fig:orientation}.}
  \label{fig:rIIMove}
\end{figure}

Recall the algebra $\mathcal{E}_i=\AZ(e_1,e_2)$ with $\lvert
e_1\rvert=i-1$, $\lvert e_2\rvert=i$, $\partial e_2=e_1,$ and $\partial
e_1=0$. Define the vector space map $H:S(\AZ')\to S(\AZ')$ by

\[H(w)=\begin{cases}
  0&w\in\AZ'\\
  0&w=Qe_2R\text{ with }Q\in\AZ',R\in S(\AZ')\\
  (-1)^{\lvert Q\rvert+1}Qe_2R&w=Qe_1R\text{ with }Q\in\AZ',R\in
  S(\AZ').
\end{cases}\]

Note that either crossing $a$ or $b$ is a positive crossing, so
$\partial b=-a+v$, where $v$ is a sum of terms in the $a_i$ and
$t_i^{\pm1}$. Define $\Phi_0:\AZ\to S_{\lvert b\rvert}(\AZ')$ by
\[\Phi_0(w)=\begin{cases}
  e_2&w=b\\
  -e_1+v&w=a\\
  w&\text{otherwise.}
\end{cases}\] 
\cite{EtnyreInvariants} tells us $\Phi_0$ is a
grading-preserving elementary isomorphism. Inductively define maps
$\Phi_i$ on the generators of $\AZ$ by:
\[\Phi_i(w)=\begin{cases}
  b_i+H(\partial'b_i-\Phi_{i-1}\partial b_i)&w=b_i\\
  \Phi_{i-1}(w)&\text{otherwise.}
\end{cases}\]
In \cite{EtnyreInvariants}, it is shown that $\Phi:=\Phi_s$ is a DGA
isomorphism between $\AZ$ and $S_{\lvert b\rvert}(\AZ')$.

If there is an augmentation $\epsilon'$ on $S(\AZ')$, then
$\epsilon=\epsilon'\Phi$ is an augmentation on $\AZ$. One can check
that
\begin{equation} \label{eqn:augComp}
  \epsilon(a_i)=\epsilon'(a_i),\quad\quad\epsilon(a)=\epsilon'(v),\quad\quad\epsilon(b)=\epsilon'(e_2).
\end{equation}
Recall that if $\lvert e_2\rvert=0$, then $\epsilon'(e_2)$ can be chosen
arbitrarily.

Analogous to the result for the $\integers/2$ case in
\cite{SabloffAug}, we have:
\begin{lem}\label{lem:3.2equiv} After a
  type II Reidemeister move involved in making a dip in a plat
  diagram, suppose $\epsilon(b_i)$ has been determined for $i<j$. Then
  \[\epsilon(b_j)=\epsilon'(b_j)-\sum_p\epsilon(b_j;Q_paR_p)(-1)^{\lvert
    \Phi(Q_p)\rvert}\epsilon(Q_pbR_p)\] for
  $Q_p,R_p\in\AZ'$ such that $\partial
  b_j=P+\sum_p\epsilon(b;Q_paR_p)Q_paR_p$ where $P$ is the sum of the
  terms in $\partial b_j$ which do not contain $a$.
\end{lem}

\begin{proof}
  We know
  \[\Phi(b_i)=b_i+H(\partial'b_i-\Phi\partial b_i).\]

  We will prove the result by inducting on $j$. For the base case, suppose
  $j=1$. Since $\partial$ lowers height, we know $\partial
  b_1\in\AZ(a,b,a_1,\ldots,a_r)$ and
  $\partial'b_1\in\AZ(a_1,\ldots,a_r)$. By Lemma \ref{lem:3.1equiv},
  we know if $P$ is the sum of terms in $\partial b_1$ which do not
  contain $a$, then $\partial b_1$ has the form
  \[\partial b_1=P+\sum_p\epsilon(b_1;Q_paR_p)Q_paR_p,\]
  where $Q_p,R_p\in\AZ(a_1,\ldots,a_r)$. Therefore
  \begin{align*}
    H(\partial'b_1-\Phi\partial b_1)&=H\left(\partial'b_1-\Phi\left(P+\sum_p\epsilon(b_1;Q_paR_p)Q_paR_p\right)\right)\\
    &=H\left(\partial'b_1-\Phi(P)-\sum_p\epsilon(b_1;Q_paR_p)Q_p(-e_1+v)R_p\right).
  \end{align*}
  We know $\partial'b_1\in\AZ(a_1,\ldots,a_r)$, so
  $H(\partial'b_1)=0$. Since $P\in\AZ(b,a_1,\ldots,a_r)$, we know
  $\Phi(P)\in\AZ(e_2,a_1,\ldots,a_r)$ and so $H(\Phi(P))=0$. Thus
  \begin{align*}
    H(\partial'b_1-\Phi\partial b_1)&=-\sum_p\epsilon(b_1;Q_paR_p)H(Q_p(-e_1+v)R_p)\\
%     &=\sum_p\epsilon(b_1;Q_paR_p)H(Q_pe_1R_p+Q_pvR_p)\\
%     &=\sum_p\epsilon(b_1;Q_paR_p)H(Q_pe_1R_p)\\
    &=\sum_p(-1)^{\lvert Q_p\rvert+1}\epsilon(b_1;Q_paR_p)Q_pe_2R_p.
  \end{align*}
  So
  \begin{align*}
    \epsilon(b_1)&=\epsilon'(\Phi(b_1))\\
    &=\epsilon'(b_1+H(\partial'b_1-\Phi\partial b_1))\\
    &=\epsilon'(b_1)+\epsilon'\left(\sum_p(-1)^{\lvert Q_p\rvert+1}\epsilon(b_1;Q_paR_p)Q_pe_2R_p\right)\\
%     &=\epsilon'(b_1)+\sum_p(-1)^{\lvert Q_p\rvert+1}\epsilon(b_1;Q_paR_p)\epsilon'(Q_p)\epsilon'(e_2)\epsilon'(R_p)\\
%     &=\epsilon'(b_1)+\sum_p(-1)^{\lvert Q_p\rvert+1}\epsilon(b_1;Q_paR_p)\epsilon'(\Phi(Q_p))\epsilon'(\Phi(b))\epsilon'(\Phi(R_p))\\
    &=\epsilon'(b_1)-\sum_p(-1)^{\lvert
      Q_p\rvert}\epsilon(b_1;Q_paR_p)\epsilon(Q_pbR_p).
  \end{align*}

  Since
  \begin{align*}
    \Phi(b_1)&=b_1+H(\partial'b_1-\Phi\partial b_1)\\
    &=b_1-\sum_p(-1)^{\lvert Q_p\rvert}\epsilon(b_1;Q_paR_p)Q_pe_2R_p,
  \end{align*}
  we have also shown that $e_1$ does not appear in $\Phi(b_1)$.

  Now suppose the equation is satisfied for $b_i$ and that $e_1$ does
  not appear in $\Phi(b_i)$ for $i<j$. As before, since $\partial$ is
  height decreasing, $\partial
  b_j\in\AZ(a,b,a_1,\ldots,a_r,b_1,\ldots,b_{j-1})$ and
  $\partial'b_j\in\AZ(a_1,\ldots,a_r,b_1,\ldots,b_{j-1})$. By Lemma
  \ref{lem:3.2equiv} we know that if $P$ is the sum of terms in
  $\partial b_j$ which do not contain $a$, then
  \[\partial b_j=P+\sum_p\epsilon(b_j;Q_paR_p)Q_paR_p,\]
  where $Q_p,R_p\in\AZ(a_1,\ldots,a_r,b_1,\ldots,b_{j-1})$. By the
  inductive assumption, $\Phi(b_i)$ does not contain $e_1$ for $i<j$
  and so $\Phi(Q_p),\Phi(R_p)$, and $\Phi(P)$ do not contain $e_1$. So
  \begin{align*}
    H(\Phi(Q_paR_p))&=H(\Phi(Q_p)(-e_1+v)\Phi(R_p))\\
    % &=H(-\Phi(Q_p)e_1\Phi(R_p)+\Phi(Q_p)v\Phi(R_p))\\
    &=(-1)^{\lvert\Phi(Q_p)\rvert}\Phi(Q_p)e_2\Phi(R_p).
  \end{align*}
  Therefore
  \[H(\partial'b_j-\Phi\partial b_j)
  % &=-\sum_pH(\Phi(\epsilon(b_j;Q_paR_p)Q_paR_p))\\
  =-\sum_p(-1)^{\lvert\Phi(Q_p)\rvert}\epsilon(b_j;Q_paR_p)\Phi(Q_p)e_2\Phi(R_p).\]
  Thus $\Phi(b_j)=b_j+H(\partial'b_j-\Phi\partial b_j)$ does not
  contain $e_1$.

  We then see
  \begin{align*}
    \epsilon(b_j)&=\epsilon'\Phi(b_j)\\
    &=\epsilon'(b_j+H(\partial'b_j-\Phi\partial b_j))\\
    % &=\epsilon'(b_j)+\epsilon'\left(-\sum_p(-1)^{\lvert
    %     Q_p\rvert}\epsilon(b_j;Q_paR_p)\Phi(Q_p)e_2\Phi(R_p)\right)\\
    % &=\epsilon'(b_j)-\sum_p(-1)^{\lvert
    %   Q_p\rvert}\epsilon(b_j;Q_paR_p)\epsilon'(\Phi(Q_p))\epsilon'(\Phi(b))\epsilon'(\Phi(R_p))\\
    &=\epsilon'(b_j)-\sum_p(-1)^{\lvert
      \Phi(Q_p)\rvert}\epsilon(b_j;Q_paR_p)\epsilon(Q_pbR_p),
  \end{align*}
  as desired.
\end{proof}

Therefore, after a type II move involved in making a dip, if
$\epsilon(b_i)$ has been determined for $i<j$, then
\[\epsilon(b_j)=\epsilon'(b_j)-\sum(-1)^{\lvert
  \Phi(Q_p)\rvert}\epsilon(b_j;Q_paR_p)\epsilon(Q_pbR_p),\] where the
sum is over totally augmented disks with positive corner at $b_j$ and
a negative corner at $b$.

\bigskip
\section{Augmentation to Ruling}\label{sec:augRuling}
In this section, we will use a construction similar to that of
Sabloff's in \cite{SabloffAug} to construct a $\rho$-graded normal
ruling from a $\rho$-graded augmentation to a fixed field $F$. This
shows the forward direction of Theorem \ref{thm:main}. Suppose that
$D$ is the front diagram of a Legendrian knot $\Lambda$ in plat
position. By the discussion in \S \ref{sec:basepoints} we can assume
that there are base points $*_1,\ldots,*_m$, one on each right cusp,
labeled from top to bottom corresponding to $t_1,\ldots,t_m$. Let
$\epsilon':\calA_Z\to F$ be a $\rho$-graded augmentation of the DGA
$(\calA_Z,\partial)$ over $Z=\integers[t_1^{\pm1},\ldots,t_m^{\pm1}]$
of $D$. (Note that then $\epsilon'(t)=\prod_{i=1}^m\epsilon'(t_i)$ for
the corresponding augmentation over $\integers[t,t^{-1}]$.)  We will
construct a $\rho$-graded normal ruling for the knot diagram while
simultaneously extending the augmentation to an augmentation
$\epsilon$ of the dipped diagram by adding one dip at a time from left
to right. We will add base points to the diagram as we go to simplify
the augmentation.

Start the ruling at the left of the diagram, pairing strands $2k$ and
$2k-1$ for $1\leq k\leq m$. We will extend the ruling from left to
right along the diagram such that Property (R), stated below, is
satisfied. We can ensure Property (R) is satisfied because when
introducing new crossings in the creation of the dips, the
$a/b$-lattices, we get to choose where the augmentation sends the
crossings in the $b$-lattice. We have enumerated the conditions we
will need to check to ensure we end up with a $\rho$-graded
augmentation of the dipped diagram and a $\rho$-graded normal ruling.

\bigno {\bf Property (R):} At any dip, the generator $a^j_{rs}$ is
augmented if and only if the strands $r$ and $s$ are paired in the ruling between
$c_j$ and $c_{j+1}$.

\bigskip Recall that the crossings from the resolution of the right
cusps are labeled $q_1,\ldots, q_m$ from top to bottom and that the
remaining crossings are labeled $c_1,\ldots,c_n$ from left to
right. Also, the strands are labeled from bottom to top. It will also
be important to recall that the orientation signs at positive original
crossings are given by the left most diagram in Figure
\ref{fig:orientation}, while orientation signs at positive crossings
in the $a/b$-lattices are given in the middle diagram.

\medno We will inductively define augmentations on partially dipped
diagrams by adding dips one at a time from left to right and defining
augmentations on these diagrams. In particular, if $\epsilon_j$ is an
augmentation on the diagram with dips added up to the crossing $c_j$,
we will extend the ruling and construct $\epsilon_{j+1}$, an
augmentation on the diagram with dips added up to the crossing
$c_{j+1}$:
\begin{enumerate}
\item Extend the ruling over $c_j$ by a switch if
  $\epsilon_j(c_j)\neq0$ and just to the left of $c_j$, the ruling
  matches configuration (a), (b), or (c) in Figure
  \ref{fig:config}. Otherwise, no switch.
\item Consult Figure \ref{fig:basepoints} to determine whether any
  base points will be added between $c_j$ and $c_{j+1}$. For each
  added base point, follow the strand it will end up on to the right
  all the way to a right cusp and add a base point $*_\alpha$ at the
  right cusp. Fix $\epsilon_{j+1}(t_\alpha)=-1$ and recall from
  \S\ref{sec:basepoints} that we must then set
  $\epsilon_{j+1}(t_i)=-\epsilon_j(t_i)$, where $*_i$ is the base
  point already at the right cusp ($1\leq i\leq m$). Move the base
  point $*_\alpha$ along the strand to between $c_j$ and $c_{j+1}$,
  modifying the augmentation on any crossing the base point goes over
  or under by a factor of $-1$ according to Remark
  \ref{rmk:basepointMoving}.
\item Place a dip between crossings $c_j$ and $c_{j+1}$, making sure
  to place the dip so that the new base points are to the right if
  they end up in the dip according to Figure \ref{fig:basepoints} and
  to the left if not. Between each Reidemeister II move involved in
  making the dip:
  \begin{enumerate}
  \item Extend the augmentation $\epsilon'$ of the DGA of the diagram
    before the Reidemeister II move to an augmentation $\epsilon$ of
    the DGA of the new diagram satisfying Property (R) by defining
    $\epsilon$ on the two new crossings by Figure \ref{fig:basepoints}
    and modifying $\epsilon$ from $\epsilon'$ by Lemma
    \ref{lem:3.2equiv}.
  \item Move base points to location specified by Figure
    \ref{fig:basepoints} and modify $\epsilon$ using Remark
    \ref{rmk:basepointMoving}.
  \end{enumerate}
\end{enumerate}

Note that $\epsilon_{j+1}$ will agree with $\epsilon_j$ on the diagram
to the left of $c_j$ though, according to Lemma \ref{lem:3.2equiv},
they may differ on $c_{j+1},\ldots,c_n$.

When we complete this process and have a fully dipped diagram, the
augmentation $\epsilon_n=\epsilon$ is a $\rho$-graded augmentation of
the dipped diagram, and we have a normal ruling of the original
diagram. We will also see that the resulting augmentation has
restrictions on what $\epsilon(t)$ equals depending on whether $\rho$
is even or odd, yielding Theorem \ref{thm:rhoEven} and Theorem
\ref{thm:rhoOdd}.

\begin{figure}
  \labellist
  \small\hair 3pt
  \pinlabel {$-$(a)} [b] at 272 1863
  \pinlabel $a_1$ [tl] at 147 1687
  \pinlabel $a_2$ [tl] at 212 1750
  \pinlabel $a$ [b] at 271 1800
  \pinlabel $a^{-1}$ [tr] at 360 1719
  \pinlabel $aa_1$ [tl] at 440 1687
  \pinlabel $aa_2$ [tl] at 505 1750
  
  \pinlabel {$+$(a)} [b] at 972 1863
  \pinlabel $a_1$ [tl] at 847 1687
  \pinlabel $a_2$ [tl] at 912 1750
  \pinlabel $a$ [b] at 971 1800
  \pinlabel $a^{-1}$ [tr] at 1060 1719
  \pinlabel $aa_1$ [tl] at 1140 1687
  \pinlabel $aa_2$ [tl] at 1205 1750
   
  \pinlabel {$-$(b)} [b] at 271 1528
  \pinlabel $a_1$ [tl] at 147 1421
  \pinlabel $a_2$ [tl] at 181 1388
  \pinlabel $a$ [b] at 271 1504
  \pinlabel $a^{-1}a_1a_2^{-1}$ [tr] at 332 1421
  \pinlabel $a^{-1}$ [tr] at 392 1357
  \pinlabel $a^{-1}a_1$ [tl] at 436 1424
  \pinlabel $aa_2$ [tl] at 472 1389
   
  \pinlabel {$+$(b)} [b] at 972 1528
  \pinlabel $a_1$ [tl] at 847 1421
  \pinlabel $a_2$ [tl] at 881 1388
  \pinlabel $a$ [b] at 971 1504
  \pinlabel $a^{-1}a_1a_2^{-1}$ [tr] at 1032 1421
  \pinlabel $a^{-1}$ [tr] at 1092 1357
  \pinlabel $a^{-1}a_1$ [tl] at 1136 1424
  \pinlabel $aa_2$ [tl] at 1172 1389
  
  \pinlabel {$-$(c), product of signs of $a^{j-1}_{Li}$ and $a^{j-1}_{i+1,K}$ is $+1$} [b] at 271 1232
  \pinlabel {$+$(c), product of signs of $a^{j-1}_{Li}$ and $a^{j-1}_{i+1,K}$ is $-1$} [b] at 271 1192
  \pinlabel $a_1$ [tl] at 148 1084
  \pinlabel $a_2$ [tl] at 181 1051
  \pinlabel $a$ [b] at 271 1101
  \pinlabel $a^{-1}a_1a_2^{-1}$ [tr] at 392 1018
  \pinlabel $a^{-1}$ [tr] at 331 1084
  \pinlabel $a^{-1}a_1$ [tl] at 436 1085
  \pinlabel $aa_2$ [tl] at 472 1051
   
  \pinlabel {$-$(c), product of signs of $a^{j-1}_{Li}$ and $a^{j-1}_{i+1,K}$ is $-1$} [b] at 971 1232
  \pinlabel {$+$(c), product of signs of $a^{j-1}_{Li}$ and $a^{j-1}_{i+1,K}$ is $+1$} [b] at 971 1192
  \pinlabel $a_1$ [tl] at 848 1084
  \pinlabel $a_2$ [tl] at 881 1051
  \pinlabel $a$ [b] at 971 1101
  \pinlabel $a^{-1}a_1a_2^{-1}$ [tr] at 1094 1018
  \pinlabel $a^{-1}$ [tr] at 1031 1084
  \pinlabel $a^{-1}a_1$ [tl] at 1136 1085
  \pinlabel $aa_2$ [tl] at 1172 1051

  \pinlabel {(d)} [b] at 271 870
  \pinlabel $a_1$ [tl] at 147 729
  \pinlabel $a_2$ [tl] at 181 760
  \pinlabel $a$ [b] at 271 810
  \pinlabel $a_1$ [tl] at 440 696
  \pinlabel $a_2$ [tl] at 505 760

  \pinlabel {$-$(e)} [b] at 271 540
  \pinlabel $a_1$ [tl] at 147 398
  \pinlabel $a_2$ [tl] at 179 431
  \pinlabel $a$ [b] at 271 513
  \pinlabel $aa_1^{-1}a_2$ [tr] at 329 429
  \pinlabel $a_2$ [tl] at 440 431
  \pinlabel $a_1$ [tl] at 472 398
   
  \pinlabel {$+$(e)} [b] at 971 540
  \pinlabel $a_1$ [tl] at 847 398
  \pinlabel $a_2$ [tl] at 879 431
  \pinlabel $a$ [b] at 971 513
  \pinlabel $aa_1^{-1}a_2$ [tr] at 1029 429
  \pinlabel $a_2$ [tl] at 1139 431
  \pinlabel $a_1$ [tl] at 1172 398

  \pinlabel {$-$(f), product of signs of $a^{j-1}_{L,i+1}$ and $a^{j-1}_{iK}$ is $+1$} [b] at 271 247
  \pinlabel {$+$(f), product of signs of $a^{j-1}_{L,i+1}$ and $a^{j-1}_{iK}$ is $-1$} [b] at 271 207
  \pinlabel $a_1$ [tl] at 147 67
  \pinlabel $a_2$ [tl] at 180 99
  \pinlabel $a$ [b] at 271 117
  \pinlabel $aa_1a_2^{-1}$ [tr] at 394 35
  \pinlabel $a_1$ [tl] at 440 99
  \pinlabel $a_2$ [tl] at 472 67
   
  \pinlabel {$-$(f), product of signs of $a^{j-1}_{L,i+1}$ and $a^{j-1}_{iK}$ is $-1$} [b] at 971 247
  \pinlabel {$+$(f), product of signs of $a^{j-1}_{L,i+1}$ and $a^{j-1}_{iK}$ is $+1$} [b] at 971 207
  \pinlabel $a_1$ [tl] at 847 67
  \pinlabel $a_2$ [tl] at 880 99
  \pinlabel $a$ [b] at 973 122
  \pinlabel $aa_1a_2^{-1}$ [tr] at 1093 38
  \pinlabel $a_1$ [tl] at 1140 100
  \pinlabel $a_2$ [tl] at 1172 67

  \endlabellist
  \includegraphics[scale=.35]{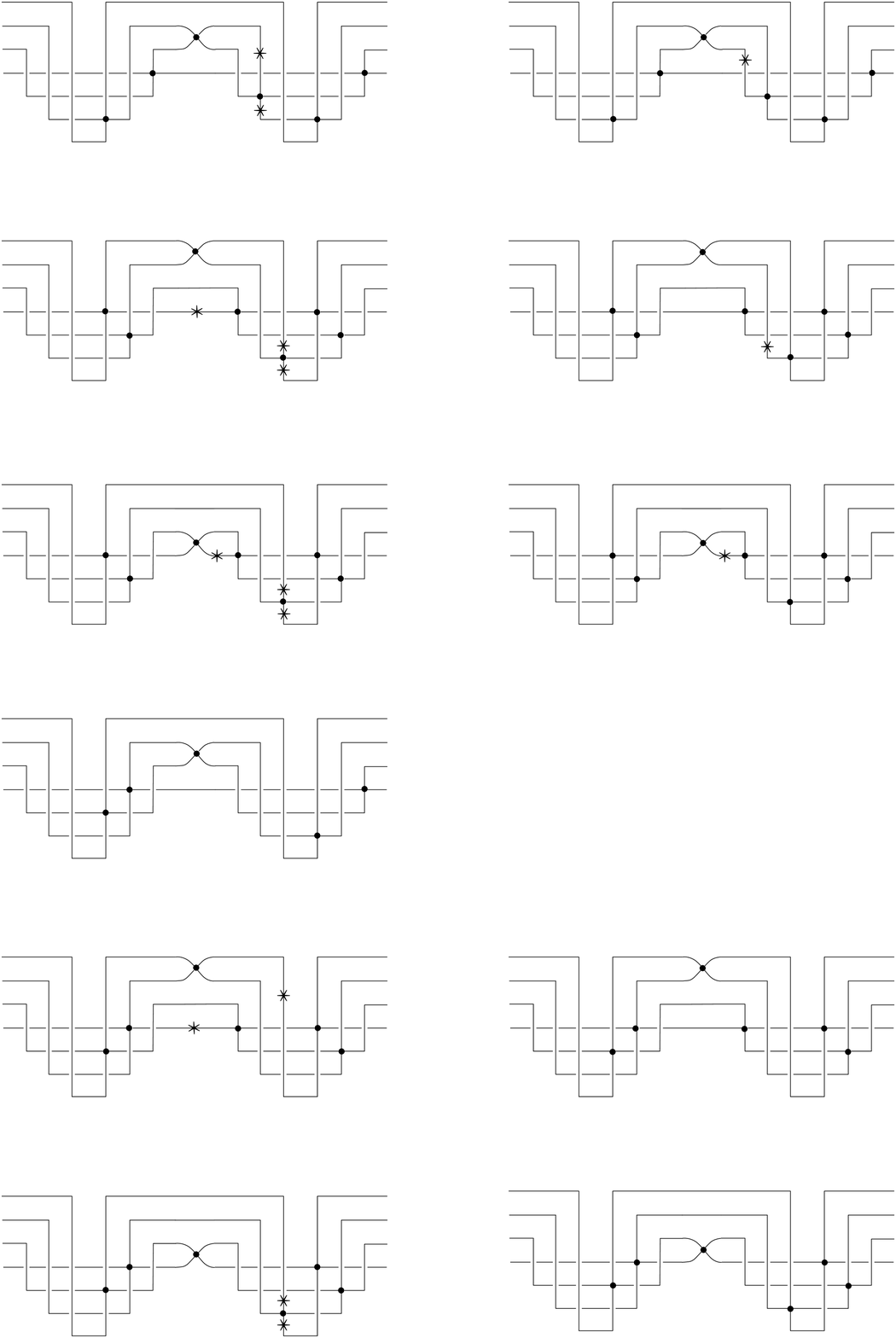}
  \caption{In the diagrams, $*$ denotes a base point. A dot denotes
    the specified crossing is augmented and the augmentation sends the
    crossing to the label. Here $-/+$(a) denotes a negative/positive
    crossing where the ruling has configuration (a) and the rest are
    defined analogously.}
  \label{fig:basepoints}
\end{figure}

For example, Figure \ref{fig:trefoilEx} gives an augmentation to
$\reals$ of the right handed trefoil and the resulting ruling and
augmentation of the dipped diagram from following this process.

\begin{figure}
  \labellist
  \small
  \pinlabel $\frac12$ [b] at 44 230
  \pinlabel $\frac12$ [b] at 447 230 
  \pinlabel $-1$ [l] at 724 308
  \endlabellist
  \includegraphics[width=4in]{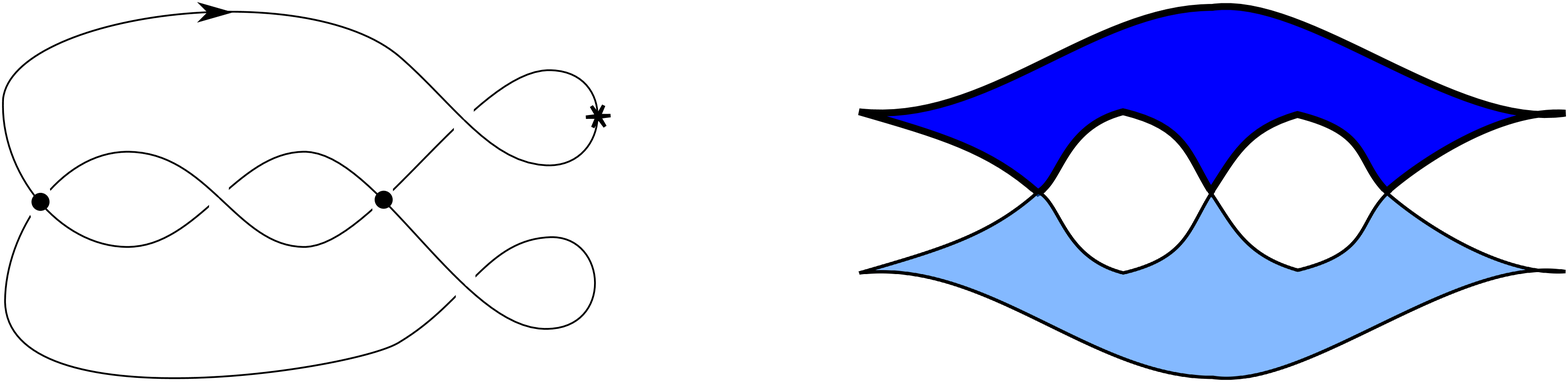}\\ \vspace{.25in}

  \labellist
  \small
  \pinlabel $1$ [tl] at 163 36
  \pinlabel $1$ [tl] at 228 100

  \pinlabel $-\frac12$ [b] at 283 154

  \pinlabel $-2$ [tr] at 377 67
  \pinlabel $-\frac12$ [tl] at 455 36
  \pinlabel $-\frac12$ [tl] at 521 100

  \pinlabel $2$ [b] at 578 154

  \pinlabel $\frac12$ [tr] at 669 67
  \pinlabel $-1$ [tl] at 748 36
  \pinlabel $-1$ [tl] at 813 100

  \pinlabel $-1$ [b] at 870 154

  \pinlabel $-1$ [tr] at 962 67
  \pinlabel $1$ [tl] at 1041 36
  \pinlabel $1$ [tl] at 1106 100
  \endlabellist

  \includegraphics[width=6.5in]{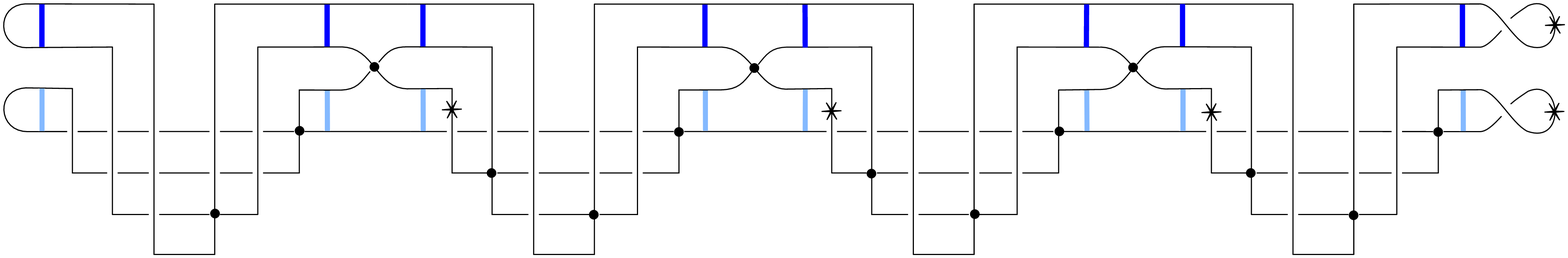}
  \caption{The top left diagram gives an augmentation of the right
    trefoil. The top right diagram gives the ruling and the bottom
    diagram gives the augmentation of the dipped diagram resulting
    from following the process of the proof. The dots denote that the
    crossing is augmented and the label on the dot gives where the
    augmentation sends the crossing. The $*$ gives the placement of
    the base points. All base points are sent to $-1$ by the
    augmentation. (In general, it may not be the case that all base
    points are sent to $-1$.)}
  \label{fig:trefoilEx}
\end{figure}

\subsection{Left cusps}
Let $\epsilon_0$ be the $\rho$-graded augmentation of the original
diagram.  We know the ruling must pair strand $2k$ with strand $2k-1$
for $1\leq k\leq m$ (where $m$ is the number of right cusps) at the
left end of the diagram. Now add a dip between the left cusps and
$c_1$. We must now extend $\epsilon_0$ to an augmentation $\epsilon_1$
of the new diagram. This will require successively extending the
augmentation $\epsilon'$ of the diagram before the Reidemeister II
move to the augmentation $\epsilon$ of the diagram after one of the
moves involved in constructing a dip. We will compute how the
augmentation $\epsilon_0$ changes as we complete each Reidemeister II
move in constructing the dip.

Consider the type II Reidemeister move which pushes strand $k$ over
strand $\ell$. We must consider the following when extending
$\epsilon'$, the augmentation before pushing strand $k$ over strand
$\ell$, to $\epsilon$, the augmentation of the resulting diagram.

\begin{enumerate}
\item We must choose $\epsilon'(e_2)$. In this case, choose
  $\epsilon'(e_2)=0$. Thus, equation \eqref{eqn:augComp} tells us
  \[\epsilon(b^0_{k\ell})=\epsilon'(e_2)=0.\]
\item By equation \eqref{eqn:augComp},
  \[\epsilon(a^0_{k\ell})=\epsilon'(v_{k\ell}),\]
  where
  \[\partial b^0_{k\ell}=a^0_{k\ell}+v_{k\ell}.\]
  From Figure \ref{fig:leftCusps}, we know $v_{k\ell}$ is a sum of
  words in $b^0_{ij}$ for $(i,j)<(k,\ell)$ and contains a $1$ if
  $(k,\ell)=(2r,2r-1)$ for some $1\leq r\leq m$. Since
  $\epsilon'(b^0_{ij})=0$ for all $(i,j)<(k,\ell)$, by step (1), we
  see that
  \[\epsilon(a^0_{k\ell})=\epsilon'(v_{k\ell})=\begin{cases}
    1&(k,\ell)=(2r,2r-1)\text{ for some }1\leq r\leq m\\
    0&\other.
  \end{cases}\]
\item We must now check whether any ``corrections'' need to be made to
  $\epsilon'$ to get $\epsilon$. In particular, whether there are any
  ``corrections'' which need to be made to $\epsilon'$ on the
  $a^0_{ij}$ generators with $(i,j)<(k,\ell)$ but $i-j\geq k-\ell$. As
  $\epsilon'(e_2)=0$, Lemma \ref{lem:3.2equiv} tells us there are no
  corrections.
\end{enumerate}

We must now check that the resulting augmentation is $\rho$-graded. We
know
\[\lvert
b^0_{2r,2r-1}\rvert=\mu(2r)-\mu(2r-1)=(\mu(2r-1)+1)-\mu(2r-1)=1\] for
$1\leq r\leq m$ and so
\[\lvert a^0_{2r,2r-1}\rvert=\lvert b^0_{2r,2r-1}\rvert-1=0\]
for $1\leq r\leq m$. So if $\epsilon'$ is $\rho$-graded, then
$\epsilon$ is also and clearly $\epsilon$ is an augmentation
satisfying Property (R).

\begin{figure}
  \labellist
  \small
  \pinlabel $b^0_{43}$ [tl] at 100 36
  \tiny
  \pinlabel $-$ [br] at 100 36
  \endlabellist
  \includegraphics[width=1.5in]{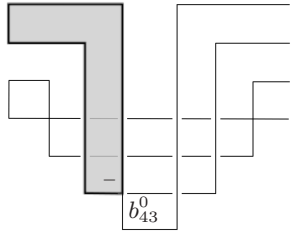}
  \caption{Shaded region gives the disk which contributes $1$ to
    $\partial b^0_{43}$.}
  \label{fig:leftCusps}
\end{figure}

\bigskip
\subsection{Extending across original crossings}
Consider the crossing $c_j$, the crossing of strands $i$ and $i+1$.
Let us extend the ruling across the crossing $c_j$ and use
$\epsilon_j$, the augmentation of the diagram with dips added up to
the crossing $c_j$, to define $\epsilon_{j+1}$, the diagram with dips
added up the crossing $c_{j+1}$. Note that $\epsilon_{j+1}$ will agree
with $\epsilon_j$ on crossings to the left of the dip added between
$c_j$ and $c_{j+1}$.

First we need to extend the ruling; extend the ruling across $c_j$ by
a switch if $\epsilon_j(c_j)\neq0$ and just to the left of $c_j$, the
ruling so far matches configuration (a), (b), or (c). Otherwise, there
is no switch. Let $1\leq L,K\leq n$ such that strand $i$ is paired
with strand $L$ and strand $i+1$ is paired with strand $K$ in
the ruling between $c_j$ and $c_{j+1}$.

We will now construct a dip between between $c_j$ and $c_{j+1}$, move
base points into place, and extend $\epsilon_j$ to an augmentation
$\epsilon_{j+1}$ in the process.

It will be useful to note that Table \ref{tab:augDisks} gives all
possibly totally augmented disks in the various configurations of the
ruling near crossings, up to base points.

% The following table gives a summary of all the possibly
% augmented disks:
\begin{table} 
  \caption{All possible totally augmented disks.}
  \label{tab:augDisks}
  \[\begin{array}{|l|l|l|}
    \hline\text{Configuration of $c_j$}&\text{Positive corner}&\text{Terms in }\partial\text{ corresp. to totally aug. disks up to base pts.}\\
    \hline&&\\
    \text{not augmented}& b^j_{rs},\,r,s\text{ paired, }r,s\notin\{i,i+1\}&a^{j-1}_{rs},a^j_{rs}\\
    &b^j_{\{i,L\}}&a^{j-1}_{\{i+1,L\}},a^j_{\{i,L\}}\\
    &b^j_{\{i+1,K\}}&a^{j-1}_{\{i,K\}},a^j_{\{i+1,K\}}\\[.1in]
    \hline&&\\

    \text{(a)}& b^j_{rs},\,r,s\text{ paired, }r,s\notin\{i,i+1\}&a^{j-1}_{rs},a^j_{rs}\\
    &b^j_{iL}&c_ja^{j-1}_{iL},a^j_{iL}\\
    &b^j_{i+1,L}&a^{j-1}_{iL},b^j_{i+1,i}a^j_{iL}\\
    &b^j_{Ki}&a^{j-1}_{K,i+1},a^{j-1}_{K,i+1}c_jb^j_{i+1,i}\\
    &b^j_{K,i+1}&a^j_{K,i+1}c_j,a^j_{K,i+1}\\[.1in]

    \hline&&\\
    \text{(b)}& b^j_{rs},\,r,s\text{ paired, }r,s\notin\{i,i+1\}&a^{j-1}_{rs},a^j_{rs}\\
    &b^j_{iK}&a^{j-1}_{i+1,K},c_ja^{j-1}_{iL}b^j_{LK}\\
    &b^j_{i+1,K}&a^{j-1}_{iL}b^j_{LK},a^j_{i+1,K}\\
    &b^j_{iL}&c_ja^{j-1}_{iL},a^j_{iL}\\
    &b^j_{i+1,L}&a^{j-1}_{iL},b^j_{i+1,i}a^j_{iL}\\[.1in]

    \hline&&\\
    \text{(c)}& b^j_{rs},\,r,s\text{ paired, }r,s\notin\{i,i+1\}&a^{j-1}_{rs},a^j_{rs}\\
    &b^j_{Ki}&a^{j-1}_{K,i+1},a^{j-1}_{K,i+1}c_jb^j_{i+1,i}\\
    &b^j_{Li}&a^{j-1}_{Li}b^j_{i+1,i},a^j_{Li}\\
    &b^j_{K,i+1}&a^{j-1}_{K,i+1}c_j,a^j_{K,i+1}\\
    &b^j_{L,i+1}&a^{j-1}_{Li},b^j_{LK}a^j_{K,i+1}\\[.1in]
    \hline&&\\

    \text{(d)}& b^j_{rs},\,r,s\text{ paired, }r,s\notin\{i,i+1\}&a^{j-1}_{rs},a^j_{rs}\\
    &b^j_{iL}&a^{j-1}_{i+1,L},a^j_{iL}\\
    &b^j_{K,i+1}&a^{j-1}_{Ki},a^j_{K,i+1}\\[.1in]
    \hline&&\\

    \text{(e)}& b^j_{rs},\,r,s\text{ paired, }r,s\notin\{i,i+1\}&a^{j-1}_{rs},a^j_{rs}\\
    &b^j_{iL}&a^{j-1}_{i+1,L},a^j_{iL}\\
    &b^j_{iK}&c_ja^{j-1}_{iK},a^{j-1}_{i+1,L}b^j_{LK}\\
    &b^j_{i+1,K}&a^{j-1}_{iK},a^j_{i+1,K}\\[.1in]
    \hline&&\\

    \text{(f)}& b^j_{rs},\,r,s\text{ paired, }r,s\notin\{i,i+1\}&a^{j-1}_{rs},a^j_{rs}\\
    &b^j_{Li}&a^{j-1}_{L,i+1},a^j_{Li}\\
    &b^j_{K,i+1}&a^{j-1}_{Ki},a^j_{K,i+1}\\
    &b^j_{L,i+1}&a^{j-1}_{Li}c_j,b^j_{LK}a^j_{K,i+1}\\[.1in]
    \hline

  \end{array}\]
\end{table}

\newpage Since the way we extend the ruling across $c_j$ depends on
$\epsilon_j(c_j)$ and the ruling immediately to the left of $c_j$, we
will need to consider when $\epsilon_j(c_j)=0$ and $\epsilon_j(c_j)\neq0$.

\underline{\bf (Case 1: $\epsilon_j(c_j)=0$)} In this case, extend the
ruling across $c_j$ without a switch. As with adding a dip between the
left cusps and $c_1$, we will compute how the augmentation $\epsilon'$
of the diagram before a Reidemeister II move changes to an
augmentation $\epsilon$ after each move involved in the making the
dip. Consider the type II move that pushes strand $k$ over strand
$\ell$. Let $\epsilon'$ be the augmentation on the diagram before the
move and let $\epsilon$ be the augmentation on the resulting diagram.
We will proceed as follows:
\begin{enumerate}
\item Define $\epsilon$ on the $b^j$-lattice.
\item Define $\epsilon$ on the $a^j$-lattice.
\item Make corrections to $\epsilon$ using Lemma \ref{lem:3.2equiv}.
\item Make corrections due to moving base points into place.
\end{enumerate}

Following this process, we have:
\begin{enumerate}
\item Choose $\epsilon'(e_2)=0$.
\item From equation \eqref{eqn:augComp}, we know
  \[\epsilon(a^j_{k\ell})=\epsilon'(v_{k\ell}).\]
  Since neither $c_j$ nor any crossing in the $b^j$-lattice is
  augmented, the only totally augmented disks in $v_{k\ell}$ have a
  positive corner at $b^j_{k\ell}$ and a single augmented negative
  corner in the $a^{j-1}$-lattice.

  \begin{figure}
    \labellist
    \tiny\hair 3pt
    \pinlabel $-$ [bl] at 207 37
    \pinlabel $-$ [bl] at 271 100
    \pinlabel $-$ [bl] at 338 164
    \pinlabel $+$ [br] at 485 164
    \pinlabel $+$ [br] at 517 132
    \pinlabel $+$ [br] at 581 37
    \endlabellist
    \includegraphics[width=4.5in]{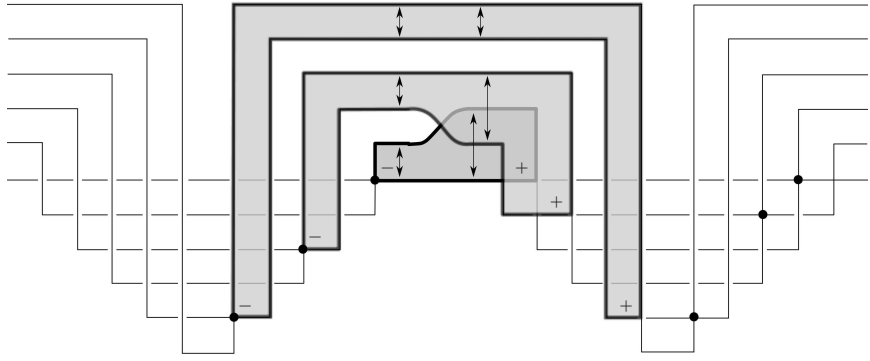}
    \caption{The disks with one negative corner in the
      $a^{k-1}$-lattice which contribute terms to the differential of
      crossings in the $b^k$-lattice if $\epsilon_j(c_j)=0$.}
    \label{fig:nonaugDisks}
  \end{figure}
 
  If such a disk exists, by Property (R), the negative corner in the
  $a^{j-1}$-lattice must be where two paired strands in the ruling
  cross as seen in Figure \ref{fig:nonaugDisks}. Since this is the
  only negative corner of the disk, we know $k$ and $\ell$ are paired
  in the ruling between $c_j$ and $c_{j+1}$ as well. So, if we recall that $a^j_{\{k,\ell\}}=a^j_{\max(k,\ell),\min(k,\ell)},$ then
  \begin{align*}
    \epsilon(a^j_{k\ell})&=\epsilon'(v_{k\ell})=\epsilon'(\epsilon(b^j_{k\ell};a^{j-1}_{k\ell})a^{j-1}_{k\ell})\\
    &=\begin{cases}
      \epsilon(a^{j-1}_{\{i,K\}})&(k,\ell)=\{i+1,K\}\\
      \epsilon(a^{j-1}_{\{i+1,L\}})&(k,\ell)=\{i,L\}\\
      \epsilon(a^{j-1}_{k\ell})&\text{ if }k,\ell\text{ paired and }k,\ell\neq i,i+1\\
      0&\other.
    \end{cases}
  \end{align*}

\item Since $\epsilon'(e_2)=0$, by Lemma \ref{lem:3.2equiv}, we know
  there are no ``corrections'' to $\epsilon(a^j_{rs})$ for
  $(r,s)<(k,\ell)$.
\item As there are no base points to move into place, no modifications
  to the augmentation are needed.
\end{enumerate}

We must now check that the resulting augmentation is
$\rho$-graded. Since $\epsilon'$ satisfies Property (R), we know
$a^{j-1}_{\{i,K\}}$, $a^{j-1}_{\{i+1,L\}}$, and $a^{j-1}_{k\ell}$ are
augmented if strands $k$ and $\ell$ are paired between $c_j$ and
$c_{j+1}$. Thus, if $\epsilon'$ is a $\rho$-graded augmentation, then
each has degree divisible by $\rho$. Since $\partial$ lowers degree by
one,
\[\lvert b^j_{\{i+1,K\}}\rvert=\lvert a^{j-1}_{\{i,K\}}\rvert+1,\quad\quad
  \lvert b^j_{\{i,L\}}\rvert=\lvert a^{j-1}_{\{i+1,L\}}\rvert+1,\quad\quad
  \lvert b^j_{k\ell}\rvert=\lvert a^{j-1}_{k\ell}\rvert+1\]
and since $\lvert a^j_{rs}\rvert=\lvert b^j_{rs}\rvert-1$,
\[\lvert a^j_{\{i+1,K\}}\rvert=\lvert a^{j-1}_{\{i,K\}}\rvert,\quad\quad
  \lvert a^j_{\{i,L\}}\rvert=\lvert a^{j-1}_{\{i+1,L\}}\rvert,\quad\quad
  \lvert a^j_{k\ell}\rvert=\lvert a^{j-1}_{k\ell}\rvert.\]
So $\epsilon$ is a $\rho$-graded augmentation satisfying Property (R)
if $\epsilon'$ is $\rho$-graded.

\bigskip
% augmented crossing case
\underline{\bf (Case 2: $\epsilon_j(c_j)\neq0$)} Now suppose $c_j$ is
augmented. This breaks into six cases, one for each possible
configuration of $c_j$ seen in Figure \ref{fig:config}. In each case,
while creating the dip, we will extend the augmentation $\epsilon_j$
of the knot diagram before adding the dip between crossings $c_j$ and
$c_{j+1}$ over the dip, move the base points into place and modify the
augmentation accordingly to end up with an augmentation
$\epsilon_{j+1}$ of the modified diagram. As in the case where $c_j$
was not augmented, we will compute how the augmentation changes as we
do each Reidemeister II move involved in making a dip between $c_j$
and $c_{j+1}$.

% configuration a
{\bf Configuration (a):} By considering Figure \ref{fig:basepoints},
we see that if $c_j$ is a negative crossing, we add two base points at
the right cusp to the right on strand $i+1$ and move them along strand
$i+1$ to between $c_j$ and $c_{j+1}$, modifying the augmentation on
any crossings we push the base points over/under according to Remark
\ref{rmk:basepointMoving}. Note that as we are moving two base points
along the same strand, no modification of the augmentation is
necessary. If $c_j$ is a positive crossing, we add one base point on
strand $i$ and follow the same process, though, in this case,
modification of the augmentation by a factor of $-1$ on the crossings
we push the base point over/under is necessary by Remark
\ref{rmk:basepointMoving}. Note that whether $c_j$ is a positive or
negative crossing, one base point will be to the left of the dip we
are adding, and, if $c_j$ is a negative crossing, we will also have
one base point to the right.

Consider the Reidemeister II move where strand $k$ is pushed over
strand $\ell$. Let $\epsilon'$ be the augmentation on the diagram
before the move and let $\epsilon$ be the augmentation of the diagram
after. Note that by our strand labeling convention $L<i<i+1<K$.

As before, we must consider the following:

$(k,\ell)<(i+1,i)$:
\begin{enumerate}
\item Choose $\epsilon'(e_2)=0$.
\item We know $\epsilon(a^j_{k\ell})=\epsilon'(v_{k\ell}).$ If $k\neq
  i,i+1$, then Table \ref{tab:augDisks} tells us
  \[\epsilon'(v_{k\ell})=\epsilon'(\epsilon(b^j_{k\ell};a^{j-1}_{k\ell})a^{j-1}_{k\ell}).\]
  So, in this case, $v_{k\ell}$ has a totally augmented disk if and
  only if $\epsilon'(a^{j-1}_{k\ell})\neq0$ if and only if $k$ and
  $\ell$ are paired between $c_{j-1}$ and $c_{j+1}$ by Property
  (R). Otherwise $(k,\ell)=(i+1,L)$ or $(k,\ell)=(i,L)$. In these
  cases
  \begin{align*}
    \epsilon(a^j_{iL})&=\epsilon'(v_{iL})\\
    &=\begin{cases}
      \epsilon'(\epsilon(b^j_{iL};c_ja^{j-1}_{iL})c_ja^{j-1}_{iL})&c_j\text{ negative crossing}\\
      \epsilon'(\epsilon(b^j_{iL};c_ja^{j-1}_{iL})t_\alpha^{\pm1}c_ja^{j-1}_{iL})&c_j\text{
        positive crossing}
    \end{cases}\\
    &=\epsilon(c_ja^{j-1}_{iL})
  \end{align*}
  and
  \begin{align*}
    \epsilon(a^j_{i+1,L})&=\epsilon'(v_{i+1,L})\\
    &=\begin{cases}
      \epsilon'(\epsilon(b^j_{i+1,L};a^{j-1}_{iL})t_\alpha^{\pm1}a^{j-1}_{iL})&c_j\text{ negative crossing}\\
      \epsilon'(\epsilon(b^j_{i+1,L};a^{j-1}_{iL})a^{j-1}_{iL})&c_j\text{
        positive crossing}
    \end{cases}\\
    &=\begin{cases}
      -\epsilon(a^{j-1}_{iL})&c_j\text{ negative crossing}\\
      \epsilon(a^{j-1}_{iL})&c_j\text{ positive crossing}.
    \end{cases}
  \end{align*}

\item Since $\epsilon'(e_2)=0$ by Lemma \ref{lem:3.2equiv}, there are
  no ``corrections'' to the augmentation of the previously constructed
  portion of the $a^j$-lattice.
\item In the case where $c_j$ is a negative crossing, according to
  Figure \ref{fig:basepoints}, we move a base point over $a^j_{i+1,L}$
  to get
  \[\epsilon(a^j_{i+1,L})=
  \left\{\begin{array}{ll}
    -(-\epsilon(a^{j-1}_{iL}))&c_j\text{ negative crossing}\\
    \epsilon(a^{j-1}_{iL})&c_j\text{ positive crossing}
  \end{array}\right\}
  =\epsilon(a^{j-1}_{iL}).\]
  Note that we do not need to move the other base points as they are
  to the left of the dip and so no more modifications are necessary.
\end{enumerate}

\medskip $(k,\ell)=(i+1,i)$:
\begin{enumerate}
\item According to Figure \ref{fig:basepoints}, choose
  $\epsilon'(e_2)=(\epsilon(c_j))^{-1}$. Then
  $\epsilon(b^j_{i+1,i})=\epsilon'(e_2)=(\epsilon(c_j))^{-1}$.

\item From looking at Table \ref{tab:augDisks}, we see that
  $v_{i+1,i}=0$ and so $\epsilon(a^j_{i+1,i})=\epsilon'(v_{i+1,i})=0$.

\item As $\epsilon'(e_2)=1$ we need to check for ``corrections.''  In
  particular, the disk in Figure \ref{fig:correctionA} contributes the
  term $a^j_{i+1,i}a^j_{iL}$ to $\partial a^j_{i+1,L}$ and is the only
  disk with negative corner at $a^j_{i+1,i}$ whose other negative
  corners are augmented since $a^j_{iL}$ is the only crossing of
  strand $L$ which is augmented by Property (R). Thus Lemma
  \ref{lem:3.2equiv} tells us
  \begin{align*}
    \epsilon(a^j_{i+1,L})&=\begin{cases}
      \epsilon'(a^j_{i+1,L})-(-1)^{\lvert t_\alpha^{\pm1}\rvert}\epsilon(a^j_{i+i,L};a^j_{i+1,i}a^j_{iL})\epsilon(t_\alpha^{\pm1}b^j_{i+1,i}a^j_{iL})&c_j\text{ negative crossing}\\
      \epsilon'(a^j_{i+1,L})-(-1)^{\vert1\rvert}\epsilon(a^j_{i+i,L};a^j_{i+1,i}a^j_{iL})\epsilon(b^j_{i+1,i}a^j_{iL})&c_j\text{
        positive crossing}
    \end{cases}\\
    &=\begin{cases}
      \epsilon(a^{j-1}_{iL})+\epsilon(t_\alpha^{\pm1}b^j_{i+1,i}a^j_{iL})&c_j\text{ negative crossing}\\
      \epsilon(a^{j-1}_{iL})+\epsilon(b^j_{i+1,i}a^j_{iL})&c_j\text{
        positive crossing}
    \end{cases}\\
    &=\epsilon(a^{j-1}_{iL})-(\epsilon(c_j))^{-1}\epsilon(c_ja^{j-1}_{iL})\\
    &=0,
  \end{align*}
  where $t_\alpha$ is associated with the base point $*$, since
  \[\epsilon(a^j_{i+1,L};a^j_{i+1,i}a^j_{iL})=\begin{cases}
    -1&c_j\text{ negative crossing}\\
    1&c_j\text{ positive crossing}.
  \end{cases}\] Thus $\epsilon$ satisfies Property (R).

  \begin{figure}
    \labellist
    \small\hair 3pt
    \pinlabel $a_{i+1,L}$ [br] at 113 69
    \pinlabel $a_{iL}$ [bl] at 145 69
    \pinlabel $a_{i+1,i}$ [tl] at 114 36
    \tiny
    \pinlabel $+$ [tl] at 113 69
    \pinlabel $-$ [tr] at 145 69
    \pinlabel $-$ [bl] at 114 36
    \endlabellist
    \includegraphics[width=1.35in]{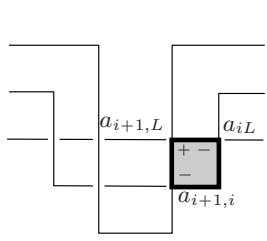}
    \caption{The disk contributing to $\partial a^j_{i+1,L}$, which
      requires ``correcting'' the augmentation. Crossings are labeled.}
    \label{fig:correctionA}
  \end{figure}

\item By Remark \ref{rmk:basepointMoving}, moving a base point over
  $a^j_{i+1,i}$ will not change the augmentation since
  $\epsilon(a^j_{i+1,i})=0$ in the case where $c_j$ is a negative
  crossing.
\end{enumerate}

\medskip $(k,\ell)>(i+1,i)$:
\begin{enumerate}
\item According to Figure \ref{fig:basepoints}, choose $\epsilon'(e_2)=0$.
\item As before, if neither strands $k$ nor $\ell$ is a crossing
  strand, then $a^j_{k\ell}$ is augmented if and only if $k$ and
  $\ell$ are paired in the ruling between $c_j$ and $c_{j+1}$. Note
  that this tells us the augmentation on the $a^j$-lattice is the same
  as the $a^{j-1}$-lattice. We do, however, see in Figure \ref{fig:aConfig}
  that there is one totally augmented disk in $v_{K,i+1}$ and two in
  $v_{Ki}$.

  \begin{figure}
    \labellist
    \small
    \pinlabel $a^{j-1}_{K,i+1}$ [tl] at 147 320
    \tiny
    \pinlabel $-$ [bl] at 147 320
    \small
    \pinlabel $c_j$ [t] at 271 433
    \tiny
    \pinlabel $-$ [b] at 271 435
    \small
    \pinlabel $b^j_{i+1,i}$ [tr] at 359 352
    \tiny
    \pinlabel $-$ [bl] at 359 352
    \pinlabel $+$ [br] at 392 352
    \small
    \pinlabel $b^j_{Ki}$ [tl] at 392 352

    \pinlabel $a^{j-1}_{K,i+1}$ [tl] at 810 320
    \tiny
    \pinlabel $-$ [bl] at 810 320
    \pinlabel $+$ [br] at 1058 354
    \small
    \pinlabel $b^j_{Ki}$ [tl] at 1058 356

    \pinlabel $a^{j-1}_{K,i+1}$ [tl] at 479 35
    \tiny
    \pinlabel $-$ [bl] at 479 35
    \small
    \pinlabel $c_j$ [t] at 604 147
    \tiny
    \pinlabel $-$ [b] at 604 149
    \pinlabel $+$ [br] at 726 35
    \small
    \pinlabel $b^j_{K,i+1}$ [tl] at 726 35
    \endlabellist

    \includegraphics[width=6in]{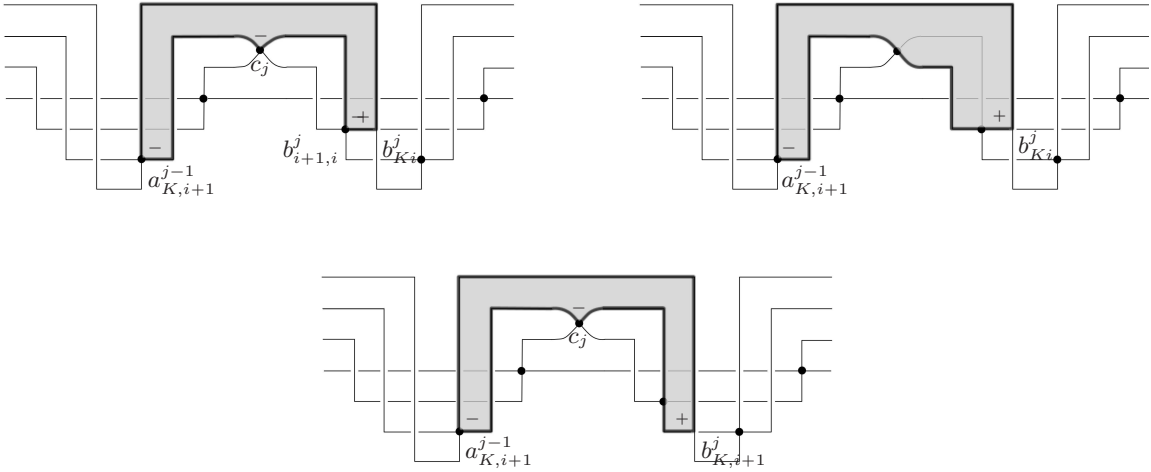}
    \caption{Totally augmented disks with one negative corner in the
      $a^{j-1}$-lattice contributing to the differential of crossings
      in the $b^j$-lattice. The crossings at corners of the disks are
      labeled.}
    \label{fig:aConfig}
  \end{figure}

  Thus
  \begin{align*}
    \epsilon(a^j_{Ki})&=\epsilon'(v_{Ki})\\
    &=\begin{cases}
      \epsilon'(\epsilon(b^j_{Ki};a^{j-1}_{K,i+1})a^{j-1}_{K,i+1})+\epsilon'(\epsilon(b^j_{Ki};a^{j-1}_{K,i+1}c_jb^j_{i+1,i})a^{j-1}_{K,i+1}c_jt_\alpha^{\pm1}b^j_{i+1,i})&c_j\text{ negative crossing}\\
      \epsilon'(\epsilon(b^j_{Ki};a^{j-1}_{K,i+1})a^{j-1}_{K,i+1}t_\alpha^{\pm1})+\epsilon'(\epsilon(b^j_{Ki};a^{j-1}_{K,i+1}c_jb^j_{i+1,i})a^{j-1}_{K,i+1}c_jb^j_{i+1,i})&c_j\text{
        positive crossing}
    \end{cases}\\
    &=\begin{cases}
      \epsilon(a^{j-1}_{K,i+1})-\epsilon(a^{j-1}_{K,i+1}c_jb^j_{i+1,i})&c_j\text{ negative crossing}\\
      -\epsilon(a^{j-1}_{K,i+1})+\epsilon(a^{j-1}_{K,i+1}c_jb^j_{i+1,i})&c_j\text{
        positive crossing}
    \end{cases}\\
    &=0,
  \end{align*}
  since
  \[\epsilon(a^j_{K,i+1}c_jb^j_{i+1,i})=\epsilon(a^{j-1}_{K,i+1}c_j)(\epsilon(c_j))^{-1}=\epsilon(a^{j-1}_{K,i+1}).\]
  And,
  \begin{align*}
    \epsilon(a^j_{K,i+1})&=\epsilon'(v_{K,i+1})\\
    &=\begin{cases}
      \epsilon'(\epsilon(b^j_{K,i+1};a^{j-1}_{K,i+1}c_j)a^{j-1}_{K,i+1}c_jt_\alpha^{\pm1}t_\beta^{\pm1})&c_j\text{ negative crossing}\\
      \epsilon'(\epsilon(b^j_{K,i+1};a^{j-1}_{K,i+1}c_j)a^{j-1}_{K,i+1}c_j)&c_j\text{
        positive crossing}
    \end{cases}\\
    &=\epsilon(a^{j-1}_{K,i+1}c_j).
  \end{align*}

\item Since $\epsilon'(e_2)=0$, by Lemma \ref{lem:3.2equiv}, no
  ``corrections.''
\item By Figure \ref{fig:basepoints}, no base points to move.
\end{enumerate}

If $\epsilon'$ is a $\rho$-graded augmentation, then
$\rho\big\vert\lvert c_j\rvert$ since $c_j$ is augmented. Thus, the
ruling is $\rho$-graded so far. We see that $\lvert
b^j_{i+1,i}\rvert=\mu(i+1)-\mu(i)=\lvert c_j\rvert$ and, since
$\partial$ lowers degree by one,
\begin{align*}
  &\lvert a^j_{K,i+1}\rvert=\lvert b^j_{K,i+1}\rvert-1=\lvert a^{j-1}_{K,i+1}\rvert\\
  &\lvert a^j_{iL}\rvert=\lvert b^j_{iL}\rvert-1=\lvert
  a^{j-1}_{iL}\rvert.
\end{align*}
As in the nonaugmentated case, if strands $k$ and $\ell$ are paired in
the ruling between $c_{j-1}$ and $c_{j+1}$, then $a^{j-1}_{k\ell}$ is
augmented and $\lvert a^j_{k\ell}\rvert=\lvert
a^{j-1}_{k\ell}\rvert$. So $\epsilon$ is a $\rho$-graded augmentation
which satisfies Property (R).

\bigskip
% configuration b
{\bf Configuration (b):} Now suppose the ruling has configuration (b)
near $c_j$. Note that with our strand assignments
$i+1>i>L>K$. According to Figure \ref{fig:basepoints}, if $c_j$ is a
negative crossing, then follow strand $K$ to the right to a right cusp
and add a base point and follow strand $i+1$ to the right to a right
cusp and add two base points. Move these base points back along their
respective strands to between $c_j$ and $c_{j+1}$, modifying the
augmentation according to Remark \ref{rmk:basepointMoving}. If $c_j$
is a positive crossing, then follow strand $i$ to the right to a right
cusp, add a base point, and move it back to between $c_j$ and
$c_{j+1}$, modifying the augmentation as necessary.

As before, we will compute how the augmentation $\epsilon_j$ changes
as we complete Reidemeister II moves involved in the construction of a
dip, to yield the extended augmentation $\epsilon_{j+1}$.

Consider the augmentation $\epsilon$ extension of the augmentation
$\epsilon'$ where strand $k$ is pushed over strand $\ell$ in the
creation of a dip between $c_j$ and $c_{j+1}$.

$(k,\ell)<(L,K)$: This case follows in the way of the first case of
configuration (a) so that setting $\epsilon'(e_2)=0$, we transfer the
augmentation on the $a^{j-1}$-lattice to that $a^j$-lattice.

\medskip $(k,\ell)=(L,K)$:
\begin{enumerate}
\item According to Figure \ref{fig:basepoints}, set
  $\epsilon'(e_2)=(\epsilon(c_ja^{j-1}_{iL}))^{-1}\epsilon(a^{j-1}_{i+1,K})$
  to obtain
  $\epsilon(b^j_{LK})=\epsilon'(e_2)=(\epsilon(c_ja^{j-1}_{iL}))^{-1}\epsilon(a^{j-1}_{i+1,K})$.
\item We see that $\epsilon'(v_{LK})=0$, since $K$ and $L$ are neither
  paired nor crossing strands in the ruling between $c_j$ and
  $c_{j+1}$. Thus
  \[\epsilon(a^j_{LK})=\epsilon'(v_{LK})=0.\]
\item There are no ``corrections'' as any disk in the $a_j$-lattice
  with negative corner at $a^j_{LK}$ must have an augmented negative
  corner of the form $a^j_{L*}$, but strand $L$ is paired with strand
  $i$ in the ruling between $c_j$ and $c_{j+1}$, so the only such
  crossing has not been made in the dip yet.
\item No base points to move, so no corrections.
\end{enumerate}

\medskip $(L,K)<(k,\ell)<(i+1,i)$:
\begin{enumerate}
\item According to Figure \ref{fig:basepoints}, set $\epsilon'(e_2)=0$.
\item In Figure \ref{fig:bConfig}, we see all the totally augmented disks
  contributing to $v_{k\ell}$ in $\partial b^j_{k\ell}$.

  \begin{figure}
    \labellist
    \small\hair 1pt
    \pinlabel $a^{j-1}_{i+1,K}$ [tr] at 146 704
    \tiny
    \pinlabel $-$ [bl] at 148 706
    \small
    \pinlabel $b^j_{iK}$ [tl] at 361 704
    \tiny
    \pinlabel $+$ [br] at 359 706

    \small
    \pinlabel $a^{j-1}_{iL}$ [tr] at 889 671
    \tiny
    \pinlabel $-$ [bl] at 891 673
    \pinlabel $-$ [t] at 983 785
    \small
    \pinlabel $c_j$ [b] at 981 788
    \pinlabel $b^j_{LK}$ [tr] at 1039 704
    \tiny
    \pinlabel $-$ [bl] at 1041 706
    \small
    \pinlabel $b^j_{iK}$ [tl] at 1070 704
    \tiny
    \pinlabel $+$ [br] at 1071 708

    \small
    \pinlabel $a^{j-1}_{iL}$ [tr] at 180 361
    \tiny
    \pinlabel $-$ [bl] at 182 363
    \small
    \pinlabel $c_j$ [b] at 272 475
    \tiny
    \pinlabel $-$ [t] at 272 473
    \small
    \pinlabel $b^j_{iL}$ [tl] at 361 361
    \tiny
    \pinlabel $+$ [br] at 361 363

    \small
    \pinlabel $a^{j-1}_{iL}$ [tr] at 889 361
    \tiny
    \pinlabel $-$ [bl] at 891 363
    \small
    \pinlabel $b^j_{LK}$ [tr] at 1040 391
    \tiny
    \pinlabel $-$ [bl] at 1042 393
    \small
    \pinlabel $b^j_{i+1,K}$ [tl] at 1102 392
    \tiny
    \pinlabel $+$ [br] at 1102 394

    \small
    \pinlabel $a^{j-1}_{iL}$ [tr] at 534 66
    \tiny
    \pinlabel $-$ [bl] at 536 68
    \small
    \pinlabel $b^j_{i+1,L}$ [tl] at 747 66
    \tiny
    \pinlabel $+$ [br] at 747 68

    \endlabellist

    \includegraphics[width=6.5in]{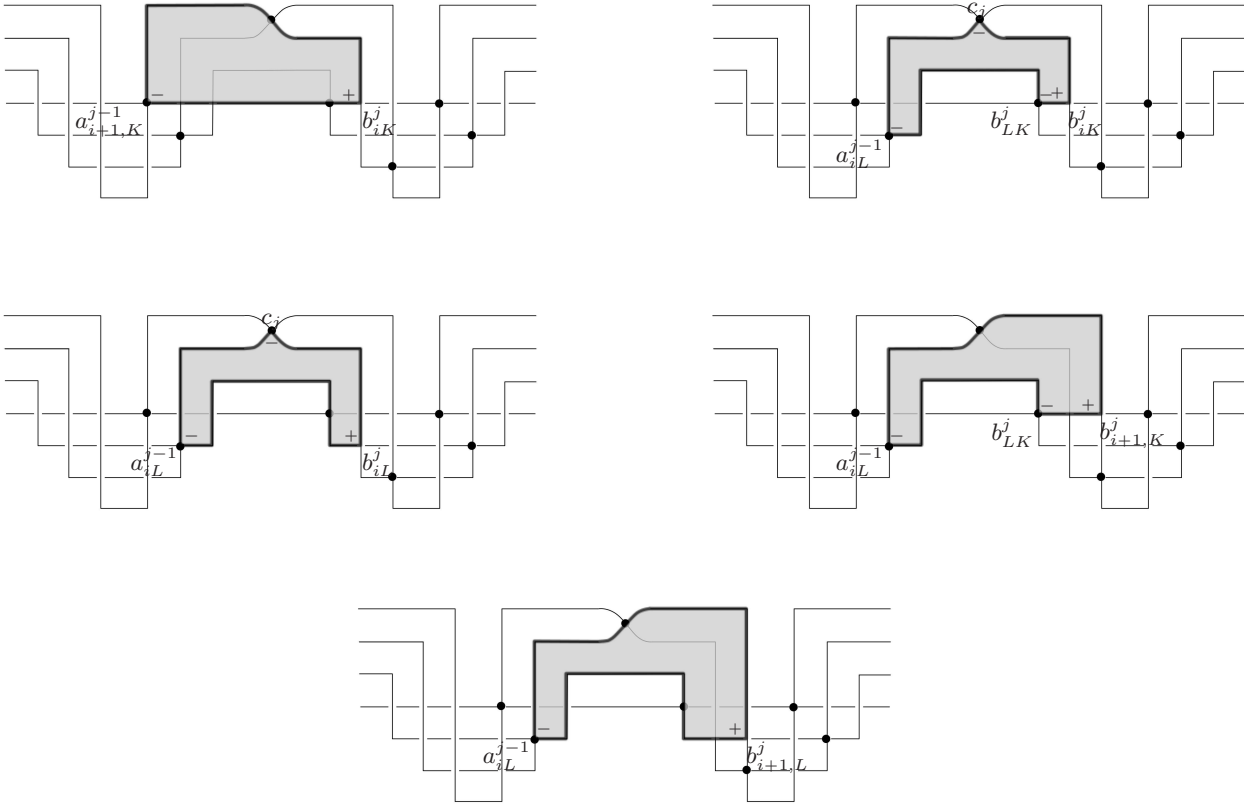}
    \caption{Totally augmented disks with one negative corner in the
      $b^{j-1}$-lattice which contribute to the differential of a
      crossing in the $a^j$-lattice. All crossings at corners of disks
      are labeled.}
    \label{fig:bConfig}
  \end{figure}

  Therefore
  \begin{align*}
    \epsilon(a^j_{iK})&=\epsilon'(v_{iK})\\
    &=\begin{cases}
      \epsilon'(\epsilon(b^j_{iK};a^{j-1}_{i+1,K})a^{j-1}_{i+1,K}t_\alpha^{\pm1})+\epsilon'(\epsilon(b^j_{iK};c_ja^{j-1}_{iL}b^j_{LK})c_ja^{j-1}_{iL}b^j_{LK})&c_j\text{ negative crossing}\\
      \epsilon'(\epsilon(b^j_{iK};a^{j-1}_{i+1,K})a^{j-1}_{i+1,K})+\epsilon'(\epsilon(b^j_{iK};c_ja^{j-1}_{iL}b^j_{LK})c_ja^{j-1}_{iL}b^j_{LK})&c_j\text{
        positive crossing}
    \end{cases}\\
    &=\begin{cases}
      -\epsilon(a^{j-1}_{i+1,K})+\epsilon(c_ja^{j-1}_{iL}b^j_{LK})&c_j\text{ negative crossing}\\
      -\epsilon(a^{j-1}_{i+1,K})-\epsilon(c_ja^{j-1}_{iL}b^j_{LK})&c_j\text{
        positive crossing}
    \end{cases}\\
    &=0,
  \end{align*}
  since
  \[\epsilon(c_ja^{j-1}_{iL}b^j_{LK})=\epsilon(c_ja^{j-1}_{iL})(\epsilon(c_ja^{j-1}_{iL}))^{-1}\epsilon(a^{j-1}_{i+1,K})=\epsilon(a^{j-1}_{i+1,K}).\]
   
  We also have
   \begin{align*}
     \epsilon(a^j_{iL})&=\epsilon'(v_{iK})=\epsilon'(\epsilon(b^j_{iL};c_ja^{j-1}_{iL})c_ja^{j-1}_{iL})\\
     &=\begin{cases}
       \epsilon(c_ja^{j-1}_{iL})&c_j\text{ negative crossing}\\
       -\epsilon(c_ja^{j-1}_{iL})&c_j\text{ positive crossing},
     \end{cases}\\[.2in]
     \epsilon(a^j_{i+1,K})&=\epsilon'(v_{i+1,K})=\epsilon'(\epsilon(b^j_{i+1,K};a^{j-1}_{iL}b^j_{LK})c_ja^{j-1}_{iL})\\
     &=\epsilon(a^{j-1}_{iL}b^j_{LK})\\
     &=\epsilon(a^{j-1}_{iL})(\epsilon(c_ja^{j-1}_{iL}))^{-1}\epsilon(a^j_{i+1,K}),\\[.2in]
     \epsilon(a^j_{i+1,L})&=\epsilon'(v_{i+1,L})=\epsilon'(\epsilon(b^j_{i+1,L};a^{j-1}_{iL})a^{j-1}_{iL})=\epsilon(a^{j-1}_{iL}).
   \end{align*}

 \item Since $\epsilon'(e_2)=0$, by Lemma \ref{lem:3.2equiv}, there
   are no ``corrections.''
 \item Note that if $c_j$ is a negative crossing, according to Figure
   \ref{fig:basepoints}, we need to move two base points over
   $a^j_{i+1,K}$ and $a^j_{i+1,L}$, so no changes. However, if $c_j$
   is a positive crossing, then we need to move one base point over
   $a^j_{iK}$ and $a^j_{iL}$ to give $\epsilon(a^j_{iK})=0$ and
   \[\epsilon(a^j_{iL})=
   \left\{\begin{array}{ll}
     \epsilon(c_ja^{j-1}_{iL})&c_j\text{ negative crossing}\\
     -(-\epsilon(c_ja^{j-1}_{iL}))&c_j\text{ positive crossing}
   \end{array}\right\}
   =\epsilon(c_ja^{j-1}_{iL}).\]
  
 \end{enumerate}

 \medskip $(k,\ell)=(i,i+1)$:
 \begin{enumerate}
 \item According to Figure \ref{fig:basepoints}, set
   $\epsilon'(e_2)=(\epsilon(c_j))^{-1}$ and so
   $\epsilon(b^j_{i+1,i})=(\epsilon(c_j))^{-1}$.
 \item As before,
   $\epsilon(a^j_{i+1,i})=\epsilon'(v_{i+1,i})=\epsilon'(0)=0$.
 \item We do have one correction: the disk $a^j_{i+1,i}a^j_{iL}$ in
   $\partial a^j_{i+1,L}$. Lemma \ref{lem:3.2equiv} tells us
   \begin{align*}
     \epsilon(a^j_{i+1,L})&=\begin{cases}
       \epsilon'(a^j_{i+1,L})-(-1)^{\lvert t_\alpha^{\pm1}\rvert}\epsilon(a^j_{i+1,L};a^j_{i+1,i}a^j_{iL})\epsilon(t)_\alpha^{\pm1}b^j_{i+1,i}a^j_{iL})&c_j\text{ negative crossing}\\
       \epsilon'(a^j_{i+1,L})-(-1)^{\lvert1\rvert}\epsilon(a^j_{i+1,L};a^j_{i+1,i}a^j_{iL})\epsilon(b^j_{i+1,i}a^j_{iL})&c_j\text{
         positive crossing}
     \end{cases}\\
     &\begin{cases}
       \epsilon(a^{j-1}_{iL})+\epsilon(t_\alpha^{\pm1}b^j_{i+1,i}a^j_{iL})&c_j\text{ negative crossing}\\
       \epsilon(a^{j-1}_{iL})-\epsilon(b^j_{i+1,i}a^j_{iL})&c_j\text{
         positive crossing}
     \end{cases}\\
     &=0,
   \end{align*}
   since
   \[\epsilon(b^j_{i+1,i}a^j_{iL})=(\epsilon(c_j))^{-1}\epsilon(c_ja^{j-1}_{iL})=\epsilon(a^{j-1}_{iL}).\]

 \item As $\epsilon(a^j_{i+1,i})=0$, no corrections are needed when
   moving the base point $*_\alpha$ over $a^j_{i+1,i}$.
 \end{enumerate}

 If $\epsilon'$ is a $\rho$-graded augmentation, then
 $\rho\big\vert\lvert c\rvert$ for all augmented crossings $c$. We see
 that $\lvert b^j_{i+1,i}\rvert=\mu(i+1)-\mu(i)=\lvert c_j\rvert$ and,
 since $\partial$ lowers degree by one,
 \begin{align*}
   \lvert b^j_{LK}\rvert&=\mu(L)-\mu(K)\\
   &=\mu(L)-\mu(i)+\mu(i)-\mu(i+1)+\mu(i+1)-\mu(K)\\
   &=-\lvert a^{j-1}_{iL}\rvert-\lvert c_j\rvert+\lvert
   a^{j-1}_{i+1,K}\rvert,\\[.1in]
   \lvert a^j_{i+1,K}\rvert&=\lvert b^j_{i+1,K}\rvert-1=\lvert a^{j-1}_{i+1,K}\rvert,\\[.1in]
   \lvert a^j_{iL}\rvert&=\lvert b^j_{iL}\rvert-1=\lvert
   a^{j-1}_{iL}\rvert.
 \end{align*}
 Since $\epsilon'$ satisfies Property (R) on the $a^{j-1}$-lattice, we
 know $\epsilon$ is a $\rho$-graded augmentation which satisfies
 Property (R). In fact, $\epsilon$ is just $a^j_{i+1,i}, a^j_{iL}$
 augmented with the rest of the augmentation on the $a^j$-lattice
 transferred from the $a^{j-1}$-lattice.

 {\bf Configuration (c), (d), (e), (f):} Similarly, one can extend
 $\epsilon_j$ over a crossing $c_j$ with the ruling having
 configuration (c), (d), (e), or (f) near $c_j$ to an augmentation
 $\epsilon_{j+1}$ satisfying Property (R) by defining it on new
 crossings as specified in Figure \ref{fig:basepoints}. We omit the
 calculations.

\subsection{Right cusps}
By construction and Lemma \ref{lem:3.2equiv}, $\epsilon=\epsilon_n$ is
an augmentation. In this section, we will show that we do in fact have
a ruling. Recall that $q_1,\ldots,q_m$ are the crossings at the right
cusps numbered from top to top. Then
\[\partial q_k=t_k^{\pm1}+a^n_{2m-2k+2,2m-2k+1}\]
for $1\leq k\leq m$, where the power of $t_k$ is determined by the
orientation of the knot at the right cusp, since strands $2m-2k+2$ and
$2m-2k+1$ are incident to the $k$-th right cusps from the
bottom. Since $\epsilon$ is an augmentation,
\begin{align*}
  0&=\epsilon\partial q_k\\
  &=\epsilon(t_k^{\pm1}+a^n_{2m-2k+2,2m-2k+1})\\
  &=(\epsilon(t_k))^{\pm1}+\epsilon(a^n_{2m-2k+2,2m-2k+1}).
\end{align*}
Since $0\neq\epsilon(t)=\prod_{i=1}^s\epsilon(t_i)$,
\[\epsilon(a^n_{2m-2k+2,2m-2k+1})=-(\epsilon(t_k))^{\pm1}\neq0.\]
Since $\epsilon$ satisfies Property (R), this tells us strands
$2m-2k+2$ and $2m-2k+1$ are paired at the right cusps for all $1\leq
k\leq m$ and so this construction does give a ruling.

This concludes the proof of the forward direction of Theorem
\ref{thm:main}. This construction also gives restrictions on
$\epsilon(t)$ for any augmentation $\epsilon$. In particular, the
final statement in Theorem \ref{thm:main}:

\begin{thm}\label{thm:rhoEven}
  If $\rho$ is even with $\rho\big\vert2r(K)$, then any $\rho$-graded
  augmentation $\epsilon$ satisfies $\epsilon(t)=-1$.
\end{thm}

\begin{proof}
  Consider the associated $\rho$-graded ruling. If $\rho$ is even,
  then any $\rho$-graded ruling is only switched at crossings $c_k$
  with $\rho\big\vert\lvert c_k\rvert$ and so $2\big\vert\lvert
  c_k\rvert$. Thus any paired strands in the ruling have opposite
  orientation. If strand $i$ is oriented to the right, we assign that
  portion of the ruling, the sign $+1$ and if it is instead oriented
  to the left, we assign $-1$. Define $\s(i,k)$ to be the sign for
  strands $i>j$ paired in the ruling between $c_k$ and $c_{k+1}$. Note
  that this sign can only change going over a switched crossing.

  For example, if we have the trefoil with the orientation given in
  Figure \ref{fig:orientedTrefoil}, then
  \[\begin{array}{l||cc|cc|cc|cc}
    k&0&0&1&1&2&2&3&3\\\hline
    i&4&2&4&2&4&3&4&2\\\hline
    \s(i,k)&+1&-1&+1&-1&+1&-1&+1&-1
  \end{array}\]

  \begin{figure}
    \includegraphics[width=2.5in]{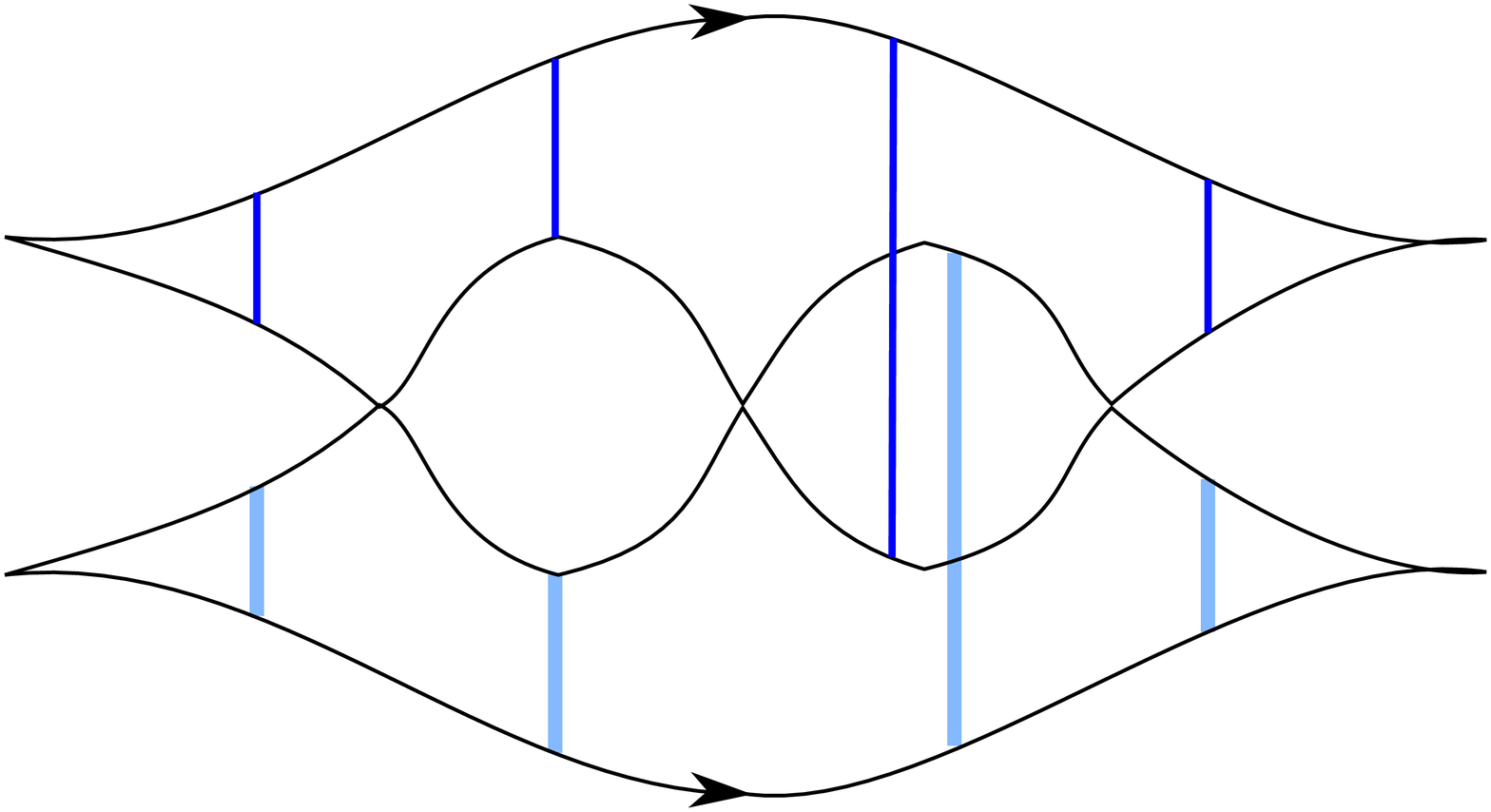}
    \caption{Oriented right handed trefoil with a normal graded ruling
      indicated.}
    \label{fig:orientedTrefoil}
  \end{figure}

  Given a $\rho$-graded ruling with $\rho$ even, we also see that we
  cannot have switched crossings which are negative crossings. So all
  switched crossings have one of the configurations appearing in
  Figure \ref{fig:abcConfig}.

  \begin{figure}
    \vspace{.2in}
    \labellist
    \small
    \pinlabel {$+$(a)} [b] at 81 170
    \pinlabel {$\s(k,K)=-1$} [t] at 81 -10
    \pinlabel {$\s(k,i)=+1$} [t] at 81 -50
    \pinlabel {$+$(a)} [b] at 297 170
    \pinlabel {$\s(k,K)=+1$} [t] at 297 -10
    \pinlabel {$\s(k,i)=-1$} [t] at 297 -50

    \pinlabel {$+$(b)} [b] at 515 170
    \pinlabel {$\s(k,i+1)=+1$} [t] at 515 -10
    \pinlabel {$\s(k,i)=+1$} [t] at 515 -50
    \pinlabel {$+$(b)} [b] at 737 170
    \pinlabel {$\s(k,i+1)=-1$} [t] at 737 -10
    \pinlabel {$\s(k,i)=-1$} [t] at 737 -50

    \pinlabel {$+$(c)} [b] at 959 170
    \pinlabel {$\s(k,L)=-1$} [t] at 959 -10
    \pinlabel {$\s(k,K)=-1$} [t] at 959 -50
    \pinlabel {$+$(c)} [b] at 1179 170
    \pinlabel {$\s(k,L)=+1$} [t] at 1179 -10
    \pinlabel {$\s(k,K)=+1$} [t] at 1179 -50
    \endlabellist

    \includegraphics[width=6in]{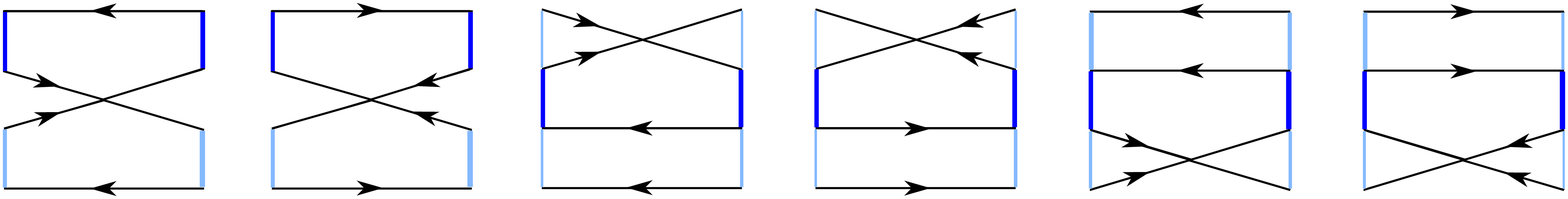}
    \vspace{.35in}
    \caption{All possible ruling configurations and orientations near
      a crossing which is switched in a $\rho$-ruling when $\rho$ is
      even with signs of ruled pairs given with our strand labeling
      convention.}
    \label{fig:abcConfig}
  \end{figure}

  Note that in these switch configurations the signs of ruling pairs
  do not change. Thus, each ruling path is an oriented unknot. The
  important part of this is that if a ruling pair has sign $+1$,
  respectively $-1$, at the left cusp, then it has sign $+1$,
  respectively $-1$, at the right cusp.

  We will show that for any $k$ such that $0\leq k\leq n$,
  \begin{equation}\prod(\epsilon(a^k_{rs}))^{\s(r,k)}=1,\label{eqn:product}\end{equation}
  where the product is taken over all paired strands $r$ and $s$ in
  the ruling between $c_k$ and $c_{k+1}$.

  Clearly this is true for $k=0$. Induct on $k$. Suppose equation
  \eqref{eqn:product} is true for $k-1$. We will show that equation
  \eqref{eqn:product} holds for $k$. If the ruling is not switched at
  $c_k$, then the result is clear. If $c_k$ has configuration type
  $+(a)$, then, by Figure \ref{fig:basepoints},
  \[\epsilon(a^k_{rs})=\begin{cases}
    \epsilon(c_k)\epsilon(a^{k-1}_{K,i+1})&(r,s)=(K,i+1)\\
    \epsilon(c_k)\epsilon(a^{k-1}_{iL})&(r,s)=(i,L)\\
    \epsilon(a^{k-1}_{rs})&\other
  \end{cases}\] and
  \[\s(K,k)=-\s(i,k),\quad \s(r,k)=\s(r,k-1)\] for all strands $r$ and
  $s$ paired in the ruling between $c_k$ and $c_{k+1}$. Thus
  \begin{align*}
    \prod_{r,s}(\epsilon(a^k_{rs}))^{\s(r,k)}&=(\epsilon(c_k)\epsilon(a^{k-1}_{K,i+1}))^{\s(K,k)}(\epsilon(c_k)\epsilon(a^{k-1}_{iL}))^{\s(i,k)}\prod_{(r,s)\neq(K,i+1),(i,L)}(\epsilon(a^{k-1}_{rs}))^{\s(r,k)}\\
    &=(\epsilon(c_k))^{-\s(i,k)}(\epsilon(a^{k-1}_{K,i+1}))^{\s(K,k)}(\epsilon(c_k))^{\s(i,k)}(\epsilon(a^{k-1}_{iL}))^{\s(i,k)}\prod_{(r,s)\neq(K,i+1),(i,L)}(\epsilon(a^{k-1}_{rs}))^{\s(r,k-1)}\\
    &=(\epsilon(a^{k-1}_{K,i+1}))^{\s(K,k-1)}(\epsilon(a^{k-1}_{iL}))^{\s(i,k-1)}\prod_{(r,s)\neq(K,i+1),(i,L)}(\epsilon(a^{k-1}_{rs}))^{\s(r,k-1)}\\
    &=\prod_{r,s}(\epsilon(a^{k-1}_{rs}))^{\s(r,k-1)}\\
    &=1.
  \end{align*}
  Similarly, we can see the same is true if $c_k$ has configuration
  $+(b)$ or $+(c)$ since $\s(r,k)=\s(r,k-1)$ for all strands $r$ and
  $s$ which are paired in the ruling between $c_k$ and $c_{k+1}$.

  In particular, the result is true for $k=n$. Since
  \[\partial
  q_\ell=t_\ell^{\s(2m-2\ell+2,2m-2\ell+1)}+a^n_{2m-2\ell+2,2m-2\ell+1},\]
  we know
  \[0=\epsilon\partial
  q_\ell=(\epsilon(t_\ell))^{\s(2m-2\ell+2,n)}+\epsilon(a^n_{2m-2\ell+2,2m-2\ell+1})\]
  for all $1\leq\ell\leq m$. Thus
  \[\epsilon(t_\ell)=-(\epsilon(a^n_{2m-2\ell+2,2m-2\ell+1}))^{\s(2m-2\ell+2,n)}\]
  and so, if $s$ is the number of base points, then
  \begin{align*}
    \epsilon(t)&=\prod_{\ell=1}^s\epsilon(t_\ell)=(-1)^{s-m}\prod_{\ell=1}^m\left(-(\epsilon(a^n_{2m-2\ell+2,2m-2\ell+1}))^{\s(2m-2\ell+2,n)}\right)\\
    &=(-1)^s\prod_{\ell=1}^m(\epsilon(a^n_{2m-2\ell+2,2m-2\ell+1}))^{\s(2m-2\ell+2,n)}\\
    &=(-1)^s\\
    &=-1
  \end{align*}
  as by Lemma \ref{lem:oddNumBasepts} we know we have an odd number of
  base points.
\end{proof}

Recall that we add an even number of base points if a crossing $c_k$
has configuration (d), (e), (f), or not augmented, two for each $-$(a)
crossing, an odd number for each $+$(a), $\pm$(b), $\pm$(c), and one
for each right cusp. Thus, to show there are an odd number of base
points, it suffices to show the following: (The following argument was
communicated to the author by Lenhard Ng.)

\begin{lem}\label{lem:oddNumBasepts}
  If $c$ gives the number of right cusps, $s$ is the number of
  switches in the ruling, and $a_-$ is the number of $-$(a) crossings,
  then
  \[c+s+a_-\equiv1\mod2.\]
\end{lem}

\begin{proof}
  We will prove this result by showing each of the following
  statements:
  \begin{align}
    &tb+r\equiv\#\text{ components}\mod2\label{state1}\\
    &tb\equiv c+cr\mod2\label{state2}\\
    &cr\equiv s\mod2\label{state3}\\
    &r\equiv a_-\mod2\label{state4}
  \end{align}
  where $r$ is the rotation number and $cr$ is the number of
  crossings. Note that if we add these four equations together, we get
  that
  \[c+s+a_-\equiv\#\text{ components}\mod2.\] Since in our case we
  have a knot, this gives the desired result.

  Statement \ref{state1} is a standard result. Statement \ref{state2}
  results from the fact that the Thurston-Bennequin number is the
  number of right cusps plus the number of crossings counted with
  sign. To prove statement \ref{state3}, we will count the number of
  interlaced pairs from left to right.

  We say that two pairs of points are {\bf interlaced} if we encounter
  the pairs alternately as we move vertically. In other words, if
  $a_i$ denotes one pair of companion strands and $b_i$ denotes the
  other, then they appear from top to bottom as $a_1b_1a_2b_2$.

  Note that the number of interlaced pairs does not change as we go
  from left to right over a switched crossing and changes by $\pm1$ as
  we go from left to right over a nonswitched crossing. We also know
  that we have zero interlaced pairs at the left and right
  cusps. Thus, the number of nonswitched crossings, which is equal to
  the number of crossings minus the number of switched crossings, is
  even, which gives
  \[cr\equiv s\mod2.\]

  The proof of statement \ref{state4} will be a little more
  involved. First, at any vertical segment of the dipped diagram which
  does not include a crossing, if $r$ and $s$ ($r>s$) are paired,
  assign the pair the number 0 if they are oriented the same way and
  $\s(r,k)$ as defined in Theorem \ref{thm:rhoEven} otherwise.  To any
  such vertical slice of the diagram, associate the sum of these
  numbers over the ruled pairs in that slice.
  % \[\sum_{\text{paired strands}}\frac{t-b}2.\]
  For example, Figure \ref{fig:leftTrefoil} gives the assignments for
  the given ruling of the left handed trefoil.
  
  % example computation for knot
  \begin{figure}
    \labellist
  % 0
    \pinlabel $+$ [l] at 426 1123
    \pinlabel $-$ [l] at 426 628
    \pinlabel $+$ [l] at 426 132
    \pinlabel $+1$ [t] at 426 -60
    % 1
    \pinlabel $+$ [l] at 1039 1123
    \pinlabel $-$ [l] at 1132 628
    \pinlabel $+$ [l] at 1072 132
    \pinlabel $+1$ [t] at 1070 -60
    % 2
    \pinlabel $+$ [l] at 1698 1123
    \pinlabel $-$ [l] at 1616 628
    \pinlabel $+$ [l] at 1708 132
    \pinlabel $+1$ [t] at 1680 -60
    % 3
    \pinlabel $+$ [l] at 2243 879
    \pinlabel $-$ [l] at 2149 628
    \pinlabel $+$ [l] at 2253 127
    \pinlabel $+1$ [t] at 2200 -60
    % 4
    \pinlabel $+$ [l] at 2790 879
    \pinlabel $-$ [l] at 2689 628
    \pinlabel $+$ [l] at 2790 379
    \pinlabel $+1$ [t] at 2740 -60
    % 5
    \pinlabel $+$ [l] at 3384 879
    \pinlabel $-$ [l] at 3283 628
    \pinlabel $+$ [l] at 3384 379
    \pinlabel $+1$ [t] at 3330 -60
    % 6
    \pinlabel $+$ [l] at 3935 879
    \pinlabel $-$ [l] at 3834 628
    \pinlabel $+$ [l] at 3935 379
    \pinlabel $+1$ [t] at 3880 -60
    % 7
    \pinlabel $-$ [l] at 4464 879
    \pinlabel $+$ [l] at 4362 628
    \pinlabel $-$ [l] at 4464 379
    \pinlabel $-1$ [t] at 4410 -60
    % 8
    \pinlabel $-$ [l] at 5037 1123
    \pinlabel $+$ [l] at 4941 628
    \pinlabel $-$ [l] at 5037 379
    \pinlabel $-1$ [t] at 4990 -60
    % 9
    \pinlabel $-$ [l] at 5596 1123
    \pinlabel $+$ [l] at 5489 628
    \pinlabel $-$ [l] at 5596 134
    \pinlabel $-1$ [t] at 5510 -60
    % 10
    \pinlabel $-$ [l] at 6160 1123
    \pinlabel $+$ [l] at 6034 628
    \pinlabel $-$ [l] at 6160 134
    \pinlabel $-1$ [t] at 6120 -60
    % 11
    \pinlabel $-$ [l] at 6692 1123
    \pinlabel $+$ [l] at 6692 628
    \pinlabel $-$ [l] at 6692 134
    \pinlabel $-1$ [t] at 6692 -60
    
    \endlabellist
    \includegraphics[width=6.25in]{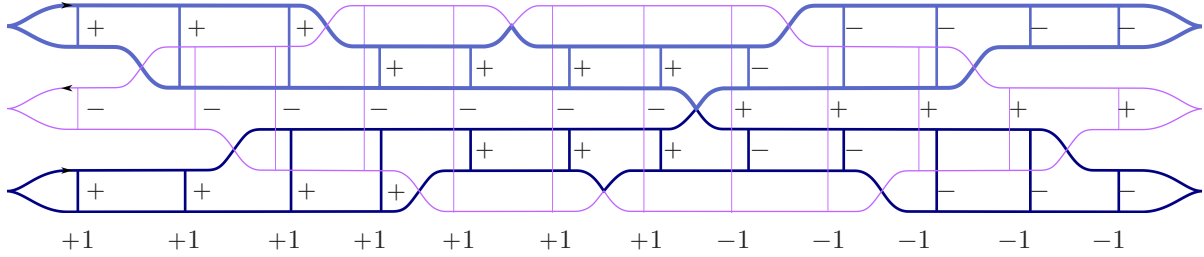}
    \vspace{.25in}
    \caption{The diagram gives a normal ruling of the left handed trefoil. At
      each vertical slice of the diagram, the paired strands in the
      ruling are decorated with $+/-$, denoting the assignment of
      $+1/-1$ to the corresponding paired strands. The number below the
      vertical slice gives the assigned sum over the ruled pairs.}
    \label{fig:leftTrefoil}
  \end{figure}

  One can check that this count goes up by $\pm2$ as you go over a
  $-(a)$ crossing and otherwise does not change. At the left cusps, we
  compute the sum to be $u_L-d_L$, where $u_L$ is the number of up
  cusps and $d_L$ the number of down. At the right cusps, we compute
  the sum to be $d_R-u_R$, where $u_R$ and $d_R$ are defined
  analogously. Therefore we have
  \[(d_R-u_R)\equiv(u_L-d_L)+2a_-\mod4.\] So
  \[2r\equiv2a_-\mod4\] and thus
  \[r\equiv a_-\mod2.\]
\end{proof}

The augmentation variety is more complicated when $\rho$ is odd. Given
a $\rho$-graded augmentation to a field $F$, once again, consider the
associated $\rho$-graded ruling.

\begin{rmk}\label{rmk:squared}
  Looking at the various configurations for the switched crossings
  (see Figure \ref{fig:basepoints}), we see that for all paired
  strands $r$, $s$ between $c_k$ and $c_{k+1}$ with $r>s$,
  \[\epsilon(a^k_{rs})=x^2\]
  for some $x\in F$ with $x\neq0$. As before, since
  \[\partial q_k=t_k^{\pm1}+a^n_{2m-2k+2,2m-2k+1}\]
  we know
  \[\epsilon(t_k)=-(\epsilon(a^n_{2m-2k+2,2m-2k+1}))^{\pm1}=-x^2\]
  for some $x\in F$ with $x\neq0$. It is then clear that, if $s$ is
  the number of base points, then
  \[\epsilon(t)=\prod_{k=1}^s\epsilon(t_k)=(-1)^{s-m}\prod_{k=1}^m\epsilon(t_k)=(-1)^sx^2=-x^2\]
  for some $x\in F$ with $x\neq0$ since, by Lemma
  \ref{lem:oddNumBasepts}, we know that $s$ is odd.
\end{rmk}

The following theorem, restated from the introduction, gives a
slightly more precise result for when there exists a $\rho$-graded
normal ruling for the diagram which is not oriented, meaning a ruling
for which not all ruling strands are oriented unknots.

\begin{thmRhoOdd}
  If $\rho$ is odd and $\rho\vert2r(\Lambda)$, then
  \[\aug_\rho(\Lambda)=\begin{cases}
    \{-x^2:x\in F^*\}&\text{ if there exists a }\rho\text{-graded normal ruling of $\Lambda$ which is not oriented}\\
    \{-1\}&\text{ if there exists a }\rho\text{-graded normal ruling of }\Lambda\text{ and all rulings are oriented}\\
    \emptyset&\text{ if there are no }\rho\text{-graded normal rulings
      of }\Lambda.
  \end{cases}\]
\end{thmRhoOdd}

\begin{proof}
  Suppose there exists a $\rho$-graded normal ruling for $\Lambda$
  which is not oriented. Fix $0\neq x\in F$. Since every ruling is
  oriented on the portion at the left cusps, for it to be an
  unoriented ruling, there has to be a crossing which takes the ruling
  from oriented to unoriented going from left to right. The only
  configurations for the ruling which do this are the crossings with
  configuration $-$(a), $-$(b), or $-$(c). Thus, a normal ruling of
  $\Lambda$ is not oriented if and only if it has a crossing with configuration
  $-$(a), $-$(b), or $-$(c). In fact, any ruling is also oriented at
  the right cusps and so must have at least two crossings where the
  ruling has configuration $-$(a),$-$(b), or $-$(c).

  Consider $\Lambda$ from the last crossing with configuration
  $-$(a),$-$(b), or $-$(c), which we will denote $c_k$, to the right
  cusps. Note that any crossing with configuration $+$(a), $+$(b),
  $+$(c), $\pm$(d), $\pm$(e), $\pm$(f), or not switched preserves the
  orientation of the paired strands in the ruling. In other words,
  whatever orientation the strands in the ruling have just to the
  right $c_k$ is the orientation they have all the way through to the
  right cusps. Let $\sigma\in S_{2m}$ be the permutation of the
  strands so that if strands $r$ and $s$ with $r>s$ are paired in the
  ruling immediately to the right of the crossing $c_k$, then strand
  $\sigma(r)$ is the strand with higher label and $\sigma(s)$ is the
  strand with lower label if we follow the ruled pair to the right
  cusps. (Note that $\sigma(r)=\sigma(s)+1$ and $2\lvert\sigma(r)$.)

  As in the $\rho$ even case, set the orientation $\s(r,j)=1$ if
  strand $r$ is oriented to the right immediately after crossing $c_j$
  and $\s(r,j)=-1$ if strand $r$ is oriented to the left for $k\leq
  j\leq n$. Labeling strands as before, this gives us
  \begin{equation}\label{eq:signs}
    \s(\max(i+1,K),k)\s(\max(i,L),k)=\begin{cases}
      +1&\text{if }c_k\text{ has configuration $-$(a)}\\
      -1&\text{if }c_k\text{ has configuration $-$(b) or $-$(c)}.
    \end{cases}
  \end{equation}

  Set $\ell_r=m+1-\frac{\sigma(r)}2$. Note that $\ell$
  is chosen so that $\sigma(r)$ and $\sigma(s)$ are the strands
  crossing at $q_\ell$. Thus
  \[\partial
  q_{\ell_r}=t_{\ell_r}^{\s(r,k)}+a^n_{\sigma(r),\sigma(s)}\] and
  so
  \[\epsilon(t_{\ell_r})=-(\epsilon(a^n_{\sigma(r),\sigma(s)}))^{\s(r,k)}\]
  since $\s(r,k)=\pm1$.

  Define $\epsilon$, an augmentation to $F$ of the DGA
  $(\calA(\Lambda'),\partial)$ of the dipped diagram $\Lambda'$ of
  $\Lambda$, satisfying Property (R), by
  \[\epsilon(c_j)=\begin{cases}
    x^{\s(K,k)}&j=k\text{ and }c_j\text{ has configuration $-$(a),$-$(c)}\\
    x^{\s(i,k)}&j=k\text{ and }c_j\text{ has configuration $-$(b)}\\
    1&\text{ if the ruling is switched at $c_j$ and }j\neq k\\
    0&\other.
  \end{cases}\] Note that Property (R) tells us that
  \[\epsilon(a^k_{rs})=\epsilon(a^n_{\sigma(r),\sigma(s)})\]
  for all strands $r$ and $s$ paired in the ruling between $c_k$ and
  $c_{k+1}$. We also note that $\epsilon$ must be a $\rho$-graded
  augmentation, since it was defined using a $\rho$-graded normal
  ruling.

  We see that if $c_j$ has configuration $-$(a), then
  \[\epsilon(a^n_{\sigma(r),\sigma(s)})=\begin{cases}
    x^{\s(K,k)}&(r,s)=(K,i+1)\\
    1&r,s\text{ paired in ruling}\\
    0&\other.
  \end{cases}\] If $c_j$ has configuration $-$(b), then
  \[\epsilon(a^n_{\sigma(r),\sigma(s)})=\begin{cases}
    x^{\s(i,k)}&(r,s)=(i,L)\\
    x^{-\s(i,k)}&(r,s)=(i+1,K)\\
    1&r,s\text{ paired in ruling}\\
    0&\other.
  \end{cases}\] Is $c_j$ has configuration $-$(c), then
  \[\epsilon(a^n_{\sigma(r),\sigma(s)})=\begin{cases}
    x^{\s(K,k)}&(r,s)=(K,i+1)\\
    x^{-\s(K,k)}&(r,s)=(L,i)\\
    1&r,s\text{ paired in ruling}\\
    0&\other.
  \end{cases}\]
  
  We know
  \begin{align*}
    \epsilon(t)&=\prod_{j=1}^s\epsilon(t_j)=(-1)^{s-m}\prod_{j=1}^m\epsilon(t_{m-j+1})\\
    &=(-1)^{s-m}\prod_{j=1}^m(-(\epsilon(a^n_{2j,2j-1}))^{\s(2j,n)})\\
    &=(-1)^s(\epsilon(a^n_{\{\sigma(K),\sigma(i+1)\}}))^{\s(\max(K,i+1),k)}(\epsilon(a^n_{\{\sigma(i),\sigma(L)\}}))^{\s(\max(i,L),k)}\mathop{\prod_{j=1}^m(\epsilon(a^n_{2j,2j-1}))^{\s(2j,n)}}_{j\neq m-\ell_{\max(K,i+1)}+1,m-\ell_{\max(L,i)}+1}\\
    &=-(\epsilon(a^n_{\{\sigma(K),\sigma(i+1)\}}))^{\s(\max(K,i+1),k)}(\epsilon(a^n_{\{\sigma(i),\sigma(L)\}}))^{\s(\max(i,L),k)}\\
    &=\begin{cases}
      -(x^{\s(K,k)})^{\s(K,k)}(x^{\s(K,k)})^{\s(i,k)}&c_k\text{ has configuration $-$(a)}\\
      -(x^{-\s(i,k)})^{\s(i+1,k)}(x^{\s(i,k)})^{\s(i,k)}&c_k\text{ has configuration $-$(b)}\\
      -(x^{\s(K,k)})^{\s(K,k)}(x^{-\s(K,k)})^{\s(L,k)}&c_k\text{
        has configuration $-$(c)}
    \end{cases}\\
    &=-x^2
  \end{align*}
  by equation \eqref{eq:signs}.
  
  By Remark \ref{rmk:squared},
  \[\aug_\rho(\Lambda)\subset\{-x^2:x\in F^*\},\]
  so $\aug_\rho(\Lambda)=\{-x^2:x\in F^*\}.$
  
  \vspace{.2in} Now suppose there exists a $\rho$-graded normal ruling
  for $\Lambda$ and all $\rho$-graded normal rulings of $\Lambda$ are
  oriented. In this case, the ruling must only have switched crossings
  with configuration $+$(a), $+$(b), $+$(c), (d), (e), or (f). Note
  that the proof of Theorem \ref{thm:rhoEven} only required this be
  the case for the ruling, so the augmentation associated to the
  normal ruling must have $\epsilon(t)=-1$ and so
  $\aug_\rho(\Lambda)=\{-1\}$.

  \vspace{.2in} If there do not exist any $\rho$-graded rulings for
  $\Lambda$, then clearly $\aug_\rho(\Lambda)=\emptyset.$
\end{proof}

\bigskip
\section{Ruling to Augmentation}\label{sec:rulingAug}
To show the backward direction of Theorem \ref{thm:main}, that given a
$\rho$-graded normal ruling of a front diagram of a Legendrian knot,
we can find a $\rho$-graded augmentation of $\A$, it suffices to show
that given a $\rho$-graded normal ruling of a front diagram, there
exists a $\rho$-graded augmentation $\epsilon$ of the dipped
diagram. We will do this by, in some sense, following the same
argument as previously, but backwards. This includes the condition
that the augmentation of the dipped diagram satisfies Property (R).

\bigskip In particular, we will be able to find an augmentation
$\epsilon$ of the dipped diagram satisfying Property (R) for which, if
a crossing $c_k$ is augmented, $\epsilon(c_k)=1$ and such that
$\epsilon(t_1\cdots t_s)=-1$ where $*_1,\ldots,*_s$ are the base
points in the final diagram.

\bigno
\subsection{Definition of Augmentation}
As previously, we can assume the base point $*$ corresponding to $t$
is in the loop of the top right cusp. We can then add one base point
to each right cusp. We will set $\epsilon(t_i)=-1$ ($1\leq i\leq m$),
this will also be true for the base points added subsequently. Note
that we will not need to do any of the ``correction'' calculations for
disks and base points as we are defining the map this way.

\subsubsection{Left}
For any ruling, at the left end of the diagram, we have strand $2k$
paired with $2k-1$ for $1\leq k\leq m$, where $m$ is the number of
right cusps. For $\epsilon$ to satisfy Property (R), we must set
\[\epsilon(b^0_{rs})=0\]
for all $k$ and $\ell$ and
\[\epsilon(a^0_{rs})=\begin{cases}
  1&\text{there exists }k\st r=2k,s=2k-1, 1\leq k\leq m\\
  0&\other.
\end{cases}\]

\subsubsection{Original crossings}
Consider a crossing $c_j$. If the ruling is switched at $c_j$, set
$\epsilon(c_j)=1$. If not, set $\epsilon(c_j)=0$.  (Note that we can
augment the switched crossings to any nonzero element of $F$ and still
get an augmentation, but we may end up with an augmentation where
$\epsilon(t)\neq-1$.)

Add base points and augment crossings in the dips, following Figure
\ref{fig:basepoints}.

\subsection{Properties of the Augmentation}
By the proof that augmentations imply rulings, $\epsilon$ is an
augmentation and by the following, the resulting augmentation
$\epsilon$ on the original undipped diagram with one base point $*$
associated to $t$ satisfies $\epsilon(t)=-1$.

Since we have set $\epsilon(t_i)=-1$ for all $1\leq i\leq s$ and Lemma
\ref{lem:oddNumBasepts} tells us $s$ is odd,
\[\epsilon(t)=\prod_{i=1}^s\epsilon(t_i)=(-1)^s=-1.\]

\bigskip
\section{Lifting Augmentations} \label{sec:lift} 
Given an augmentation to $\integers/2$ of the Chekanov-Eliashberg DGA
over $\integers/2$. We will now use constructions similar to those in
the proof of Theorem \ref{thm:main} to construct a lift of the
augmentation to an augmentation to $\integers$ of the lift of the
Chekanov-Eliashberg DGA and thus that one can construct an
augmentation to any ring. Restating from the introduction:

\begin{thmLift}
  Let $\Lambda$ be a Legendrian knot in $\reals^3$. Let
  $(\Az,\partial)$ be the Chekanov-Eliashberg DGA over $\integers/2$
  and let $(\A,\partial)$ be the DGA over $R=\integers[t,t^{-1}]$. If
  $\epsilon':\Az\to\integers/2$ is an augmentation of
  $(\Az,\partial)$, then one can find a lift of $\epsilon'$ to an
  augmentation $\epsilon:\A\to\integers$ of $(\A,\partial)$ such that
  $\epsilon(t)=-1$.
\end{thmLift}

\begin{proof}
  Recall that $\calE_i=\calA(e_1,e_2)$ where $\lvert e_1\vert=i-1$,
  $\lvert e_2\rvert=i$, $\partial(e_2)=e_1$, and $\partial(e_1)=0$ and
  $S_i(\calA_R(a_1,\ldots,a_n))=\calA_R(a_1,\ldots,a_n,e_1,e_2)$.

  Note that, for any augmentation $\epsilon$ on $\calA_R$ to
  $\integers$, there exists an augmentation $\widehat\epsilon$ on
  $S(\calA_R)$ to $\integers$ which agrees with $\epsilon$ on
  $\calA_R\subset S(\calA_r)$ and for any augmentation
  $\widehat\epsilon$ on $S(\calA_R)$ to $\integers$, there exists an
  augmentation $\epsilon$ on $\calA_R$ to $\integers$ which agrees
  with $\epsilon$ on $\calA_R\subset S(\calA_R)$. And, we have the
  analogous result for any augmentation of $\Az$. Thus, clearly one
  can find a lift $\epsilon:\calA_R\to\integers$ of
  $\epsilon':\Az\to\integers/2$ if and only if one can find a lift
  $\epsilon:S(\calA_R)\to\integers$ of
  $\epsilon:S(\Az)\to\integers/2$.

  So, if there exists a lift for $\calA$, then there exists a lift for
  any stable tame isomorphic differential graded algebra. Therefore,
  to show the result, it suffices to show there exists a lift of the
  augmentation to $\integers/2$ of differential graded algebras of
  knots in plat position. So we may assume $\Lambda$ is in plat
  position.

  Given an augmentation $\epsilon':\Az\to F$ of the
  Chekanov-Eliashberg DGA over $\integers/2$. Using Lemma
  \ref{lem:3.2equiv} modulo $2$ and the definition given in Figure
  \ref{fig:basepoints} mod $2$, we can extend $\epsilon'$ to an
  augmentation $\widehat\epsilon:\widehat\Az\to F$ of the DGA over
  $\integers/2$ for the dipped diagram of $\Lambda$. We saw that if we
  know $\widehat\epsilon(c_i)$ and the augmentation on the
  $a^i/b^i$-lattices for $i<j$, then
  \[\widehat\epsilon(c_j)\equiv\epsilon'(c_j)+\sum_{i=0}^{j-1}\sum_{k,\ell}\sum_p\widehat\epsilon(Q_pb^i_{k\ell}R_p)\mod2\]
  where, for $0\leq i<j$, $\partial
  b^i_{k\ell}=P+\sum_pQ_pa^i_{k\ell}R_p$ before passing strand $k$
  over strand $\ell$ in the creation of the dip between $c_i$ and
  $c_{i+1}$ and $P$ is the sum of terms which do not contain
  $a^i_{k\ell}$ with our labeling convention. This is the same as the
  construction introduced in \cite{SabloffAug}. From \cite{SabloffAug}
  we know that this augmentation satisfies Property (R).

  Let $(\widetilde{\calA_Z},\widetilde\partial)$ be the lift of the
  Chekanov-Eliashberg DGA over $\integers/2$ to a DGA over
  $Z=\integers[t_1^{\pm1},\ldots,t_s^{\pm1}]$ of the DGA over
  $\integers[t_1^{\pm1},\ldots,t_s^{\pm1}]$ of the dipped diagram of
  $\Lambda$. Define $\widetilde\epsilon:\widetilde{\calA_Z}\to F$ by
  \[\widetilde\epsilon(c_j)=\begin{cases}
    1&\text{ if }\widehat\epsilon(c_j)\neq0\\
    0&\other.
  \end{cases}\] on the original crossings, define $\widetilde\epsilon$
  as given by Figure \ref{fig:basepoints} for all other crossings, add
  base points where indicated in Figure \ref{fig:basepoints}, and
  define
  \[\widetilde\epsilon(t_i)=\begin{cases}
    -\widetilde\epsilon(a^n_{2m-2i+2,2m-2i+1})&\text{ if }1\leq i\leq m\\
    -1&\text{ if }m<i\leq s.
  \end{cases}\] Note that all crossings and base points are augmented
  to $0$ or $\pm1$. One can check that with this definition,
  $\widetilde\epsilon$ is an augmentation of the dipped diagram of
  $\Lambda$. Note that as the same original crossings are augmented in
  the dipped diagram, this augmentation must correspond to the same
  ruling as $\widehat\epsilon$ and by definition, satisfies Property
  (R). So, clearly,
  \[\widetilde\epsilon(c)\equiv\widehat\epsilon(c)\mod2\] for all
  crossings $c$ in the dipped diagram of $\Lambda$.

  We will use induction on $k$ to show that
  \[\prod\widetilde\epsilon(a^k_{pq})=1,\] where the product is taken
  over all paired strands $p$ and $q$, for all $1\leq k\leq n$ and
  thus, that
  \[\prod_{i=1}^s\widetilde\epsilon(t_i)=-1.\]
  Since $\widetilde\epsilon(a^0_{pq})=1$ for $(p,q)=(2m-2k+2,2m-2k+1)$
  for some $k$ such that $1\leq k\leq m$, we know
  \[\prod_{p,q}\widetilde\epsilon(a^0_{pq})=1.\]
  Looking at Figure \ref{fig:basepoints}, we see that
  \[\frac{\prod_{p,q}\widetilde\epsilon(a^k_{pq})}{\prod_{p,q}\widetilde\epsilon(a^{k-1}_{pq})}=\left\{\begin{array}{ll}
    (\widetilde\epsilon(c_{k-1}))^2&\text{ if the ruling has configuration (a) near }c_{k-1}\\
    1&\other
  \end{array}\right\}=1,\]
  since $\widetilde\epsilon(c_{k-1})=\pm1$. Thus, if
  $\prod\widetilde\epsilon(a^{k-1}_{pq})=1$, then
  $\prod\widetilde\epsilon(a^k_{pq})=1$. So, in particular,
  $\prod\widetilde\epsilon(a^n_{pq})=1$. Thus
  \begin{align*}
    \prod_{i=1}^s\widetilde\epsilon(t_i)&=(-1)^{s-m}\prod_{i=1}^m\widetilde\epsilon(t_i)=(-1)^{s-m}\prod_{i=1}^m(-\widetilde\epsilon(a^n_{2m-2i+2,2m-2i+1}))\\
    &=(-1)^s\prod_{i=1}^m\widetilde\epsilon(a^n_{2m-2i+2,2m-2i+1})=(-1)^s=-1,
  \end{align*}
  since Lemma \ref{lem:oddNumBasepts} tells $s$ is odd.

  Lemma \ref{lem:3.2equiv} in its original form also gives us a method
  to define an augmentation of the original diagram from an
  augmentation of the dipped diagram of $\Lambda$. Thus we have the
  augmentation $\epsilon:\calA_Z\to F$ of the original diagram,
  defined by
  \[\epsilon(c_j)=\widetilde\epsilon(c_j)+\sum_{i=0}^{j-1}\sum_{k,\ell}\sum_p\epsilon(b^i_{k\ell};Q'_pa^i_{k\ell}R'_p)(-1)^{\lvert
    \Phi(Q'_p)\rvert}\widetilde\epsilon(Q'_pb^i_{k\ell}R'_p)\] where,
  for $0\leq i<j$, $\partial
  b^i_{k\ell}=P+\sum_p\epsilon(b^i_{k\ell};Q'_pa^i_{k\ell}R'_p)Q'_pa^i_{k\ell}R'_p$
  before passing strand $k$ over strand $\ell$ in the creation of the
  dip between $c_i$ and $c_{i+1}$ and $P$ is the sum of terms which do
  not contain $a^i_{k\ell}$ with our labeling convention. Note that
  the ``correction'' disks in the $\integers/2$ case are the same as
  the ``correction'' disks in the
  $\integers[t_1^{\pm1},\ldots,t_s^{\pm1}]$ case, but the
  $\integers[t_1^{\pm1},\ldots,t_s^{\pm1}]$ ``correction'' disks may be
  counted with negative sign and the disk may have extra corners at
  base points. Recall that $\widetilde\epsilon(t_i)=-1$ for $m<i\leq
  s$. Thus
  \[\widetilde\epsilon(Q'_p)\equiv\widetilde\epsilon(Q_p)\mod2,\quad\quad
  \widetilde\epsilon(R'_p)\equiv \widetilde\epsilon(R_p)\mod2,\] since
  the disk which contributes $Q'_p$ (resp. $R'_p$) to the differential
  may have extra corners at base points $t_i$ for $m<i\leq s$ (base
  points not occurring at right cusps) which the disk which
  contributes $Q_p$ (resp. $R_p$) to the differential does not have.

  We will now show that $\epsilon$ is, in fact, a lift of $\epsilon'$.
  \begin{align*}
    \epsilon(c_j)&=\widetilde\epsilon(c_j)+\sum_{i=0}^{j-1}\sum_{k,\ell}\sum_p\epsilon(b^i_{k\ell};Q'_pa^i_{k\ell}R'_p)(-1)^{\lvert\Phi(Q'_p)\rvert}\widetilde\epsilon(Q'_pb^i_{k\ell}R'_p)\\
    &\equiv\widetilde\epsilon(c_j)+\sum_{i=0}^{j-1}\sum_{k,\ell}\sum_p\widetilde\epsilon(Q_pb^i_{k\ell}R_p)\mod2\\
    &\equiv\left(\epsilon'(c_j)+\sum_{i=0}^{j-1}\sum_{k,\ell}\sum_p\widehat\epsilon(Q_pb^i_{k\ell}R_p)\right)+\sum_{i=0}^{j-1}\sum_{k,\ell}\sum_p\widetilde\epsilon(Q_pb^i_{k\ell}R_p)\mod2\\
    &\equiv\epsilon'(c_j)+2\sum_{i=0}^{j-1}\sum_{k,\ell}\sum_p\widehat\epsilon(Q_pb^i_{k\ell}R_p)\mod2\\
    &\equiv\epsilon'(c_j)\mod2,
  \end{align*}
  since $\widetilde\epsilon$ is a lift of $\widehat\epsilon$. Note
  that this shows that the resulting augmentation of the DGA over
  $\integers[t_1^{\pm1},\ldots,t_s^{\pm1}]$ is a lift and so, by the
  discussion of moving and adding base points in \S
  \ref{sec:basepoints}, the augmentation of the DGA over
  $\integers[t,t^{-1}]$ is a lift as well, and
  \[\epsilon(t)=\prod_{i=1}^s\epsilon(t_i)=-1.\] And, since
  $\integers$ embeds in any ring $R$, we can also use $\epsilon'$ to
  define an augmentation $\epsilon:\A\to R$ with $\epsilon(t)=-1$.
\end{proof}

\bibliographystyle{plain}
\bibliography{AugmentationsAndRulings}{}

\end{document}